\documentclass[11pt]{amsart}

\usepackage{charter}
\usepackage[scaled=0.86]{berasans} 
\usepackage[scaled=1]{inconsolata} 
\usepackage[T1]{fontenc}
\usepackage[%
                protrusion=true,
                expansion=false,
                auto=false
                ]{microtype}

\usepackage{setspace}
\usepackage{macros}
\standardsettings
\amsbib

\newcommand{\cbar}{\w c}

\newcommand{\zetabound}{\w\zeta}
\newcommand{\xibound}{\epsilon}
\newcommand{\thetabound}{\epsilon}

\newcommand{\deltabar}{{\wbar\delta}}
\newcommand{\thetabar}{{\wbar\theta}}

\newcommand{\etaratio}{\eta}

\newcommand{\jprime}{j}
\newcommand{\mj}{j}
\newcommand{\Qpower}{p}
\newcommand{\kr}{t}
\newcommand{\an}{n}
\newcommand{\tm}{h}

\newcommand{\gepm}{\succneq}

\newcommand{\vanishinga}{Proposition \ref{propositionvanishing}(ii)}

\newcommand{\vanishingc}{Proposition \ref{propositionvanishing}(iv)}

\makeatletter
\newcommand{\opnorm}{\@ifstar\@opnorms\@opnorm}
\newcommand{\@opnorms}[1]{%
  \left|\mkern-1.5mu\left|\mkern-1.5mu\left|
   #1
  \right|\mkern-1.5mu\right|\mkern-1.5mu\right|
}
\newcommand{\@opnorm}[2][]{%
  \mathopen{#1|\mkern-1.5mu#1|\mkern-1.5mu#1|}
  #2
  \mathclose{#1|\mkern-1.5mu#1|\mkern-1.5mu#1|}
}
\makeatother

\draftfalse

\begin{document}

\title[Hausdorff dimensions of perturbations of a CIFS]{Hausdorff dimensions of perturbations of a conformal iterated function system via thermodynamic formalism}

\authortushar\authorlior\authordavid\authormariusz

\dedicatory{Dedicated to Abram Samoilovitch Besicovitch (1891--1970) and Vojt\v{e}ch Jarn\'ik (1897--1970)}

\begin{abstract}
We consider small perturbations of a conformal iterated function system produced by either adding or removing some generators with small derivative from the original. We establish a formula, utilizing transfer operators arising from the thermodynamic formalism \`a la Sinai--Ruelle--Bowen, which may be solved to express the Hausdorff dimension of the perturbed limit set in series form: either exactly, or as an asymptotic expansion. Significant applications to the dimension theory of continued fraction Cantor sets include strengthening Hensley's asymptotic formula from 1992, which improved on earlier bounds due to Jarn\'ik and Kurzweil, for the Hausdorff dimension of the set of real numbers whose continued fraction expansion partial quotients are all $\leq N$; as well as its counterpart for reals whose partial quotients are all $\geq N$ due to Good from 1941.

\end{abstract}

\maketitle


\tableofcontents

\section{Introduction}

We are approaching the close of a century of mathematics, following Hausdorff's seminal work \cite{Hausdorff}, dedicated to a panoply of measure- and dimension-theoretic research regarding the intricate fractal geometry of sets arising from classical Diophantine approximation and its manifold avatars. Abram Samoilovitch Besicovitch and Vojt\v{e}ch Jarn\'ik were among the pioneers who first broke ground at this fertile interface of algebra (number theory) and analysis (geometric measure theory), and this paper is dedicated to the beautiful vistas exposed by their mathematics. Their influential investigations have led to a blossoming area broadly known as \emph{metric Diophantine approximation}, with several connections to classical number theoretic questions, as well as more surprising links to mathematical physics, dynamical systems, fractal geometry, analytic combinatorics, computer science, wireless communication, etc. -- see \cite{CusickFlahive, Shallit, BernikDodson, Dani2, BDD, Lagarias, DodsonKristensen, Kristensen2, KSS, FSU4, Moreira2, BeresnevichVelani6, Hensley_book, IosifescuKraaikamp,WangWu} and the references therein for a sampling of such relationships.

We begin with a brief description of two theorems in this vein, which follow from our more general results that are described in later sections. Recall (e.g., \cite{Khinchin_book,Schmidt3}) that an irrational number $x$ is called \emph{badly approximable} if there exists $\epsilon >0$ such that $|x - p/q| \geq \epsilon/q^2$ for any rational $p/q$. To study Diophantine properties it suffices to consider irrationals in the unit interval, and for any irrational $x \in [0,1]$ we abbreviate its simple (or regular) continued fraction expansion as follows
\[
x=\cfrac1{a_1+\cfrac1{a_2+\cfrac1{a_3+ \ddots \;}}}
= [a_1,a_2,a_3,\dots],
\]
where the sequence of positive integers $a_i = a_i(x)$ are known as the partial quotients (or continued fraction entries/digits) of $x$.
It is well-known (\cite[Theorem 5F]{Schmidt3} or \cite[Theorem 1.9]{Bugeaud}) that $x$ is badly approximable if and only if the partial quotients in its continued fraction expansion are bounded. Thus given a finite subset $I \subset \N$ the set $\Lambda_I$ of all numbers in $[0,1]$ whose continued fraction expansions have partial quotients that belong to $I$ form a subset of the badly approximable numbers. Such sets $\Lambda_I$ are Cantor sets that may be described as conformal iterated function system (CIFS) limit sets \cite{MauldinUrbanski1,CLU}, or as cookie-cutter (Cantor) sets, after Dennis Sullivan \cite{Bedford3}. The study of their Hausdorff dimension has attracted the attention of several researchers over many decades -- for a small sampling of such work across a broad spectrum of fields see \cite{Bumby1, Bumby2, Cusick2, Cusick3, Cusick, Hensley, Hensley3, Hensley4, Hensley5, FlajoletVallee2, FlajoletVallee3, CesarattoVallee2, McMullen_conformal_3, JenkinsonPollicott2, JenkinsonPollicott3, JenkinsonPollicott4, FalkNussbaum, Kontorovich, CarminatiTiozzo3} and the references therein. In contrast, estimates and rigorous dimension computation for Cantor sets that arise from infinite subsets $I \subset \N$ (and the measure-theoretic study of limit sets of infinite CIFS, more generally) present a variety of new challenges and there is plenty left to uncover -- see \cite{Ramharter2, GardnerMauldin, FSU1, HeinemannUrbanski, MauldinUrbanski1, MauldinUrbanski4, FalkNussbaum2, CLU2, BertheLee, Gonzalez} for some progress in this vein.

Perhaps the earliest paper on the Hausdorff dimension of continued fraction Cantor sets was Jarn\'ik's paradigmatic \cite{Jarnik1}, in which he established that for every $N \geq 8$
\[
1 - \frac{4}{N \log(2)} \leq \HD(F_{\leq N}) \leq 1 - \frac{1}{8N \log(N)},
\]
where $\HD$ denotes Hausdorff dimension and $F_{\leq N}$ denotes the set of all numbers in $[0,1]$ whose continued fraction expansions have partial quotients all $\leq N$. As a corollary Jarn\'ik was able to prove his seminal result on full Hausdorff dimension of the set of badly approximable numbers\footnote{Jarn\'ik's result inspired a myriad extensions, e.g. \cite{Patterson1, Schmidt2, Dani4, FSU4, Beresnevich_BA, Simmons5}, and finding analogues of our results in any of these settings would involve tackling several new challenges.}, which may be described as an increasing union of the $F_{\leq N}$ sets.


Over two decades later in 1951, Jarn\'ik's student Jaroslav Kurzweil was able to improve the former bounds in his doctoral work by proving \cite[Theorem VIII]{Kurzweil} that 
\begin{equation}
\label{kurzweil}
1 - \frac{0.99}{N} \leq \HD(F_{\leq N}) \leq 1 - \frac{0.25}{N}
\end{equation}
for $N \geq 1000$.  This was the state of the art for the next four decades until the breakthrough work of Doug Hensley who leveraged functional analytic techniques\footnote{Hensley's approach arose from a distinguished line of research on Gauss's problem on the distribution of continued fraction partial quotients by Kuzmin, Levy, Sz\"usz, Wirsing, Babenko, Mayer and several others, see Knuth's  \cite[pp.362--366]{Knuth} for a beautiful, albeit already dated, survey.} to improve on Kurzweil's result by proving \cite{Hensley} (cf. \cite[Chapter 5]{Bugeaud}) that 
\begin{equation}
\label{hensley}
\HD(F_{\leq N}) = 1 - \frac{6}{\pi^2} \frac 1N - \frac{72}{\pi^4} \frac{\log(N)}{N^2} + O\left(\frac{1}{N^2}\right)
\end{equation}
Notice that \eqref{hensley} is stronger than \eqref{kurzweil} for sufficiently large $N$, since $.25 < 6/\pi^2 < .99$.

Hensley's haunting formula \eqref{hensley} leads to some natural questions: what does the remainder term $O(1/N^2)$ look like? Can it be written as $c/N^2 + o(1/N^2)$ for some coefficient $c$? And if so, what does the $o(1/N^2)$ here look like: do more logarithms appear? The following theorem is an example of our main result, Theorem \ref{maintheorem}, applied to the sequence of sets $(F_{\leq N})$:

\begin{theorem}
\label{theoremFleqN}
For each $p\geq 1$, the Hausdorff dimension of $F_{\leq N}$ can be estimated via the formula
\begin{equation}
\label{HDFleqN}
\HD(F_{\leq N}) = 1 + \sum_{i = 1}^{p - 1} \sum_{j = 0}^{i - 1} c_{i,j} \frac{\log^j(N)}{N^i} + O_p\left(\frac{\log^{p - 1}(N)}{N^{p}}\right),
\end{equation}
where $c_{i,j} \in \R$ are effectively computable constants. Here $O_p$ means that the implied constant of $O$ may depend on $p$.
\end{theorem}

Note that by \eqref{hensley} we have $c_{1,0} = -6/\pi^2$ and $c_{2,1} = -72/\pi^4$. Our methods yield explicit formulas for the subsequent coefficients $c_{i,j}$ (see Appendix \ref{appendix1} for some example computations),
but the formulas for $c_{2,0}$ and further coefficients depend on a certain operator $Q$ on the space of H\"older-continuous functions on $[0,1]$, defined in terms of the Gauss--Kuzmin--Wirsing operator $L$ (cf.~Theorem \ref{theoremoperatorequation}). This operator is given as a series and it appears to be quite challenging to give a closed formula for its value on explicitly given functions such as $\one(x) \df 1$. In particular, the precise formula for $c_{2,0}$ in terms of $Q$ is quite complicated; see \eqref{c20}. However, the sequence of coefficients $(c_{i,i - 1})$ turns out to have a relatively simple expression:
\begin{equation}
\label{cii-1}
c_{i,i-1} = -\frac{2^{i - 1}\cdot i^{i - 2}}{(i - 1)!} \left(\frac{6}{\pi^2}\right)^i .
\end{equation}
This includes the two coefficients $c_{1,0}$ and $c_{2,1}$ computed by Hensley.

Our techniques can be used to estimate the Hausdorff dimensions of many different sequences of sets coming from conformal iterated function systems, such as sequences of sets $(F_N)$ where each $F_N$ is specified by restricting continued fraction partial quotients to lie in some set $E_N \subset \N$, such that the sequence of characteristic functions $(\one_{E_N})$ converges pointwise to some characteristic function $\one_E$ (we denote such convergence by $E_N \to E$). In some cases, the formula for the Hausdorff dimension coming from Theorem \ref{maintheorem} ends up being far more complicated than \eqref{HDFleqN}.

For instance, consider $F_{\geq N}$, the set of elements of $[0,1]$ whose continued fraction partial quotients are all $\geq N$. The earliest estimates on the dimension of $F_{\geq N}$ were obtained in the late 1930s by Irving John (Jack) Good. Good's work \cite{Good,Good2}, which was undertaken on Besicovitch's suggestion and awarded the prestigious Smith Prize at the University of Cambridge \cite{Banks}, has since inspired a wealth of research on the dimension theory of continued fraction Cantor sets. Good proved that for $N \geq 20$
\[
\frac12 + \frac{1}{2\log(N+2)} \leq \HD(F_{\geq N}) \leq \frac12 + \frac{\log\log(N-1)}{2\log(N-1)} \cdot
\]
Almost 70 years later, Jaerisch--Kesseb\"ohmer \cite{JaerischKessebohmer} were able to prove an asymptotic improvement on Good by showing that as $N \to \infty$
\[
\HD(F_{\geq N}) - \frac12 \sim \frac{\log\log(N)}{2\log(N)}
\]
Applying our main result, Theorem \ref{maintheorem}, to the sequence of sets $(F_{\geq N})$ leads to the following strengthening of Jaerisch--Kesseb\"ohmer's result:
\begin{theorem}
\label{theoremFgeqN}
For each $p\geq 1$, the Hausdorff dimension of $F_{\geq N}$ can be estimated via the formula
\begin{equation}
\label{HDFgeqN}
\begin{split}
\HD(F_{\geq N}) &= \frac12+ \frac1{2\log(N)}\left[\log\log(N) - \log\log\log(N) + \sum_{k = 1}^\infty \sum_{\ell = 1}^k c_{k,\ell} \frac{\log^\ell \log\log(N)}{\log^k \log(N)}\right.\\
&\left.+ \sum_{i = 1}^{p - 1} \sum_{j = 1}^\infty \sum_{k = -j}^\infty \sum_{\ell = 0}^{j + k} c_{i,j,k,\ell} \frac{\log^\ell\log\log(N)}{N^i \log^j(N) \log^k\log(N)} \right] + O_p\left(\frac{\log\log(N)}{N^p \log(N)}\right)
\end{split}
\end{equation}
where $c_{k,\ell} \in \Q$ and $c_{i,j,k,\ell} \in \Q$ are appropriate constants that can be computed explicitly. For example, $c_{1,1} = -1$, $c_{2,1} = 1$, $c_{2,2} = -1/2$, $c_{3,1} = -1$, $c_{3,2} = 3/2$, $c_{3,3} = -1/3$, and $c_{1,1,-1,0} = 1/2$.
\end{theorem}

\bigskip

{\bf Outline for the sequel.} In Section \ref{sectionoperatorequation} we prove a general result in the setting of Banach spaces that will introduce the key equation leading to \eqref{HDFleqN} and \eqref{HDFgeqN}. In Section \ref{sectionPACIFS} we introduce a class of conformal iterated function systems that includes the class of Gauss IFSes, to which our results will apply. In Section \ref{sectiontheoremscase12} we state our main theorem, of which Theorems \ref{theoremFleqN} and \ref{theoremFgeqN} are special cases. In Section \ref{sectionproof} we prove this theorem, and in Section \ref{sectiongauss} we provide examples where the theorem applies, in particular Proposition \ref{propositionpolysequence} which corresponds to the above theorems. Sections \ref{sectionspectralgap}, \ref{sectionvanishing}, and \ref{sectionlogloglog} contain auxiliary results necessary for these proofs and examples. Section \ref{sectionopen} concludes with some directions for further research. Finally, in Appendix \ref{appendix1} we compute the coefficients $c_{i,i-1}$ and $c_{2,0}$ appearing in Theorem \ref{theoremFleqN}.

\bigskip

{\bf Conventions.} We use the standard Landau notation $O(\cdot)$, $\Theta(\cdot)$, as well as writing $A \lesssim B$ when $A = O(B)$ and $A \equiv_X B$ when $B - A = O(X)$. If $A \lesssim B \lesssim A$, we write $A \asymp B$. Recall that $\Q[x]$ denotes the ring of polynomials in the variable $x$ with coefficients in $\Q$. All linear operators between Banach spaces are assumed to be bounded. 
By default balls are closed, e.g. $B_\C(0,1) \df \{ z \in \C : |z| \leq 1\}$, and we denote open balls with $^\circ$, e.g. $B^\circ_\C(0,1) \df \{ z \in \C : |z| < 1\}$.
Note that we use the Iverson bracket notation in several places in the text: $[\Phi] = 1$ when $\Phi$ is true and $[\Phi] = 0$ when $\Phi$ is false. 
The notation $F_{|i}$ represents the partial derivative of the function $F$ with respect to the $i$th coordinate. Similarly $F_{|ij}$ denotes a double derivative achieved by taking a partial derivative with respect to the $i$th coordinate followed by a partial derivative with respect to the $j$th coordinate; and $F_{|ijk}$ denotes a triple derivative, etc.

\section{An abstract operator formula}
\label{sectionoperatorequation}

The idea for proving Theorem \ref{theoremFleqN} is to consider the Perron--Frobenius operators $L_\infty,L_N: C([0,1]) \to C([0,1])$ defined by the formulas
\begin{align} \label{Ldef}
L_\infty f(x) &= \sum_{\an = 1}^\infty \frac{1}{(\an+x)^2} f\left(\frac{1}{\an+x}\right)\\ \label{LNdef}
L_N f(x) &= \sum_{\an = 1}^N \frac{1}{(\an+x)^{2\delta_N}} f\left(\frac{1}{\an+x}\right)
\end{align}
where $\delta_N$ is the Hausdorff dimension of $F_{\leq N}$. Note that $L_\infty$ is the well-known Gauss--Kuzmin--Wirsing operator\footnote{The operator $L_\infty$ is variously referred to in the literature as the transfer operator for Gauss's continued fraction map, or as the Perron--Frobenius, Ruelle--Perron--Frobenius, Ruelle--Mayer, or Ruelle operator, etc. See e.g. \cite{Wirsing,Mayer3}.}. By well-known dynamical results (see \6\ref{sectionPACIFS}), the definition $\delta_N = \HD(F_{\leq N})$ can be encoded as the assertion that the spectral radius of $L_N$ is 1, which is furthermore equivalent to the assertion that $L_N$ fixes a positive function $g_N$ and the dual operator $L_N^*$ fixes a positive measure $\mu_N$. (Similarly, $L_\infty$ fixes the positive function $g(x) = 1/(1 + x)$, and its dual $L_\infty^*$ fixes $\mu$, the Lebesgue measure on $[0,1]$.) 
We wish to convert this assertion into a formula involving $L_N$, which in turn determines a relation between $N$ and $\delta_N$. 
To this end we introduce some notation.

\begin{notation*}
Let $\BB$ be a complex Banach space, let $\BB^*$ be its dual space, i.e. the Banach space of all bounded linear functionals from $\BB$ to $\C$, and let $\LL(\BB)$ denote the Banach space of all bounded linear operators from $\BB$ to $\BB$. 
Fix $f\in \BB$ and $\sigma\in \BB^*$. Then $\sigma f$ denotes the value of $\sigma$ on $f$, while $f \sigma$ (which is our shorthand for $f \otimes \sigma$) denotes the element of $\LL(\BB)$ defined as $(f\sigma)f' \df (\sigma f') f$.
Note that $f \sigma$ is a projection when $\sigma f = 1$. If $\sigma f \neq 0$, then $(\sigma f)^{-1} f \sigma$ is a projection, while if $\sigma f = 0$, then $f \sigma$ is a nilpotent operator of order $2$.
If $L \in \LL(\BB)$, then $\sigma L$, is the element of $\BB^*$ defined by the formula $(\sigma L) f \df \sigma (L f)$. The map $L^* : \sigma \mapsto \sigma L$ from $\BB^*$ to $\BB^*$ is called the dual operator of $L$. However, we avoid using the notation $L^*$ in formulas, so we will write $\sigma L$ rather than $L^* \sigma$. 
This notation is analogous to the notation used in matrix multiplication, with elements of $\BB$, $\BB^*$, and $\LL(\BB)$ corresponding to column, row, and square matrices, respectively. In particular, the associative laws 
\begin{align*}
&(\sigma L) f = \sigma (L f), &
&(f \sigma) f' = f (\sigma f'), &
&\sigma' (f \sigma) = (\sigma' f) \sigma
\end{align*} 
all hold by definition.
If $L f = f$ we call $f$ a right fixed point of $L$, and if $\sigma L = \sigma$ we call $\sigma$ a left fixed point of $L$ (or equivalently, a fixed point of the dual operator $L^*$).
\end{notation*}

\begin{theorem}
\label{theoremoperatorequation}
Let $\BB$ be a Banach space. Suppose that $L$ and $L'$ in $\LL(\BB)$ have respective right fixed points $g,g'$, and let $\mu \in \BB^*$ be a left fixed point of $L$, such that $\mu g , \mu g' \neq 0$. Let $$R \df L - c g \mu \in \LL(\BB),$$ where $c = 1/\mu g$, and let $$\Delta \df L' - L \in \LL(\BB),$$ and $\rho(R)<1$, where $\rho$ denotes the spectral radius. Also suppose that
\begin{equation}
\label{RnDeltabound}
\sum_{n = 0}^\infty \|R^n\|\cdot\|\Delta\| < 1.
\end{equation}
Then
\begin{equation}
\label{operatorequation}
\sum_{\Qpower = 0}^\infty \mu \Delta (Q\Delta)^\Qpower g = 0,
\end{equation}
where $Q \df \sum_{n = 0}^\infty R^n \in \LL(\BB)$. Note that $Q$ is well-defined since $\rho(R)<1$.
\end{theorem}

\begin{proof}
The idea is to start with the equation $L' g' = g'$, expressing the fact that $g'$ is a right fixed point for $L'$, then multiply on the left by a measure $\mu'$ to get a scalar equation, and finally rearrange to get \eqref{operatorequation}. Specifically, let
\[
\mu' \df \sum_{m = 0}^\infty \mu (L' - c g \mu)^m = \sum_{m = 0}^\infty \mu (R + \Delta)^m
\]
(We will show later that this series converges in $\BB^*$.) We have $\mu' = \mu + \mu' (L' - c g \mu)$,\Footnote{Plugging in the formula $\mu' g = \mu g$ proven below, it follows that $\mu'$ is a left fixed point of $L'$. However, this fact is irrelevant to the proof, except as an indicator that our choice of $\mu'$ is not as arbitrary as it may initially appear to be.} and thus
\[
\mu' L' - \mu' = \mu' L' - (\mu + \mu' (L' - c g \mu)) = (c \mu' g  - 1) \mu .
\]
Multiplying on the right by $g'$ and using the fact that $g'$ is fixed gives
\[
0 = \mu' L' g' - \mu' g' = (c \mu' g - 1) (\mu g')
\]
and thus since $\mu g' \neq 0$,
\[
\mu g = 1/c = \mu' g = \sum_{m = 0}^\infty \mu (R + \Delta)^m g.
\]
Now by the distributive law $\sum_{m=0}^\infty (R + \Delta)^m$ is the sum of all finite ordered products of $R$ and $\Delta$, i.e.
\[
\sum_{m=0}^\infty (R + \Delta)^m = \sum_{\Qpower = 1}^\infty \sum_{n_1,\ldots,n_\Qpower} R^{n_1} \Delta R^{n_2} \cdots R^{n_{\Qpower-1}} \Delta R^{n_\Qpower} = \sum_{\Qpower = 1}^\infty (Q\Delta)^{\Qpower-1} Q.
\]
These three series all converge absolutely since
\begin{align*}
\sum_{m=0}^\infty \| (R + \Delta)^m \| 
&\leq \sum_{\Qpower = 1}^\infty \sum_{n_1,\ldots,n_\Qpower} \|R^{n_1} \Delta R^{n_2} \cdots R^{n_{\Qpower-1}} \Delta R^{n_\Qpower}\| \\
&\leq \sum_{\Qpower=1}^\infty \left(\sum_{n=0}^\infty \|R^n\|\right)^\Qpower \|\Delta\|^{\Qpower-1} \underset{\eqref{RnDeltabound}}{<} \infty.
\end{align*}
Note that this implies that the series defining $\mu'$ converges.

Thus, we have
\[
\mu g = \sum_{\Qpower = 0}^\infty \mu (Q\Delta)^\Qpower Q g = \mu Q g + \sum_{\Qpower = 0}^\infty \mu Q\Delta (Q\Delta)^\Qpower Q g.
\]
Since $g$ is a right fixed point for $L$, we have $R g = L g - c g \mu g = g - g = 0$, and thus $Q g = g$. Similarly, since $\mu$ is a left fixed point for $L$, we have $\mu R = 0$ and $\mu Q = \mu$. 
Finally, using the identities $Q g = g$ and $\mu Q = \mu$ in the previous displayed equation, we derive \eqref{operatorequation}.
\end{proof}

\begin{remark*}
The hypothesis that $g'$ is a right fixed point of $L'$ such that $\mu g' \neq 0$ may be replaced by the hypothesis that $\mu'$ is a left fixed point of $L'$ such that $\mu' g \neq 0$, with minimal changes to the proof. (Both hypotheses are satisfied in our applications of Theorem \ref{theoremoperatorequation}.)
\end{remark*}

\begin{remark*}
\label{remarkspectral}
Since $R g = 0$ and $\mu R = 0$ (see the last paragraph of the proof above), it follows that for all $n \geq 1$ we have 
\[
L^n = (R + c g \mu)^n = R^n + (c g \mu)^n = R^n + c g \mu,
\] 
and thus
\[
Q = I + \sum_{n = 1}^\infty (L^n - c g \mu).
\]
\end{remark*}

\begin{remark*}
In the case where $L$ and $L'$ are Perron--Frobenius operators of similarity IFSes, \eqref{operatorequation} reduces to the Moran--Hutchinson equation for the latter IFS (assuming that for the former), see Proposition \ref{propositionhutchinson}.
\end{remark*}

The idea of the proof of Theorem \ref{theoremFleqN} is now to apply Theorem \ref{theoremoperatorequation} with $L = L_\infty$ as in \eqref{Ldef} and $L' = L_N$ as in \eqref{LNdef}, and then to solve the resulting formula \eqref{operatorequation} for $\delta_N$. This determines the sought-after relation between $N$ and $\delta_N$. We refer to the subsequent sections for details on how this is implemented. The idea of the proof of Theorem \ref{theoremFgeqN} is similar, except instead of taking $L$ as in \eqref{Ldef} we take $L f(x) = f(0)$, or equivalently $L = h \nu$ where $\nu$ is the Dirac point mass at $0$ and $h = \one$ (the motivation for this choice will become clear in subsequent sections, in particular Lemma \ref{lemmabetaf} and Remark \ref{remarkgaussvalues}).

\section{Point-accumulating conformal iterated function systems}
\label{sectionPACIFS}

The sets $F_{\leq N}$ and $F_{\geq N}$ can be viewed as limit sets of certain \emph{conformal iterated function systems}, or CIFSes. CIFSes were introduced by Mauldin and Urba\'nski \cite{MauldinUrbanski1} (see Appendix \ref{appendixCIFS} for their definition), and their generalizations \emph{conformal graph directed Markov systems} (CGDMSes) were studied in \cite{MauldinUrbanski2}. We will consider a certain class of CIFSes, which we define as follows.

\begin{definition}
\label{definitionPACIFS}
Fix a quintuple $(U,V,u,v,q)$ such that
\begin{itemize}
\item $U\subset\R$ is a bounded open set containing $0$;
\item $V \subset \R$ is a bounded connected open set;
\item $u:U\times V \to V$ is an real-analytic map such that
\begin{itemize}
\item the family of maps $(u_b)_{b \in U}$ defined by 
\[
V\ni x\mapsto u_b(x) = u(b,x) \in V
\] 
is either uniformly contracting (i.e. for each $(b,x) \in U\times V$, we have $|u_b'(x)| \leq \lambda$ for some uniform constant $\lambda<1$), or conjugate to a uniformly contracting family, and
\item $\bigcup_{b\in U} u_b(V)$ is precompact in $V$; and
\end{itemize}
\item $v:U\times V \to \R$ is a bounded real-analytic function and $q > 0$ is a parameter such that for all $b\in U$ and $x\in V$,
\begin{equation}
\label{vdef}
|u_b'(x)| = |b|^q e^{v(b,x)}.
\end{equation}
Note that \eqref{vdef} implies that $u_0$ is constant; for convenience, in what follows we assume that this constant is $0$, i.e. that $u_0(x) = 0$ for all $x$.
Formula \eqref{vdef} also implies that for all $x\in V$, $q$ is the order of the analytic function $b \mapsto u_b(x)$ at $0$. 
\end{itemize}
Then if $S \subset U\butnot\{0\}$ is a set whose only accumulation point, if any, is 0, then the family of maps $(u_b)_{b\in S}$ is called a \emph{point-accumulating conformal iterated function system (PACIFS)} over $(U,V,u,v,q)$. For conciseness we will usually omit ``over $(U,V,u,v,q)$'' when referring to a PACIFS $(u_b)_{b\in S}$.

The \emph{limit set} of the PACIFS $(u_b)_{b\in S}$ is denoted by $\Lambda_S$, and is the image of the projection map $\pi: \Sigma \df S^\N \to V$ defined by the formula
\[
\pi(\omega) = \lim_{n\to\infty} u_{\omega_1}\circ\cdots\circ u_{\omega_n}(x_0),
\]
where $x_0\in V$ is an arbitrary point. This limit exists and is independent of $x_0$ because the maps $(u_b)_{b\in S}$ are uniformly contracting and $V$ is connected, and $\Lambda_S \subset V$ because of the assumption that $\bigcup_{b\in U} u_b(V)$ is precompact in $V$.
\end{definition}

Note that in the sequel we will consider $\Sigma$ as a metric space under the metric $\dist(\sigma,\tau) = \lambda^{|\sigma\wedge\tau|}$, where $0 < \lambda < 1$ is fixed and $|\sigma\wedge\tau|$ is the length of the longest common initial segment $\sigma\wedge\tau$ of $\sigma$ and $\tau$.

\begin{remark*}
Our results also hold in the following more general setting: $U = \bigcup_{i=1}^k (U_i\times \{i\}) \cup F$, where each $U_i$ is a bounded open subset of $\R^{d_i}$ containing $0$, $F$ is a finite set, $V$ is a bounded connected open subset of $\R^d$, $u,v$ are analytic on each $U_i\times\{i\}\times V$ and on each $\{b\}\times V$ for $b\in F$, and each matrix $u_{b,i}'(x)$ is a similarity with dilatation factor $|u_{b,i}'(x)|$, and \eqref{vdef} is replaced by the formula $|u_{b,i}'(x)| = |b|^{q_i} e^{v_i(b,x)}$, where $b\in U_i$ and $x\in V$. This can be proven with only minor modifications (but notational complications) to the definitions and proofs.
\end{remark*}

\begin{remark*}
Our PACIFSes are not always CIFSes in the sense of \cite{MauldinUrbanski1,MauldinUrbanski2}, because they do not necessarily satisfy the open set condition (OSC), see Appendix \ref{appendixCIFS}. However, in Definition \ref{definitionOSCPACIFS} we define the class of OSC PACIFSes, and these are CIFSes in the sense of \cite{MauldinUrbanski1,MauldinUrbanski2}.
\end{remark*}

\subsection{Example of PACIFS: Gauss IFSes}
\label{subsectiongauss}
Let $U = V = (-\epsilon,1 + \epsilon)$ with $0 < \epsilon < 1$ and consider the maps $u: U\times V \to V$, $v: U\times V \to \R$ defined by
\[
u(b,x) \df \frac{b}{1+bx} , \;\;
v(b,x) = -2\log(1 + bx)
\]
and let $q=2$. Then the tuple $(U,V,u,v,q)$ satisfies the requirements of Definition \ref{definitionPACIFS} except for the requirement of uniform contraction; but after conjugating by e.g. the map $\phi(x) = 1/(1+x)$, the family $(\phi\circ u_b\circ \phi^{-1})_{b\in U}$ is uniformly contracting, so our main results will apply to this system as well. Thus, each set $E \subset \N$ corresponds to an OSC pseudo-PACIFS
\[
S(E) \df \{ 1/\an : \an \in E \}
\]
corresponding to the Gauss IFS $\{u_{1/\an}:\an\in E\}$, where
\[
u_{1/\an}(x) = \frac{1}{\an + x}\cdot
\]
For each sequence $\an_1,\an_2,\ldots\in\N$ we hae
\[
\pi(1/\an_1,1/\an_2,\ldots) = [0;\an_1,\an_2,\ldots]
\]
where $[0;n_1,n_2,\ldots]$ represents the continued fraction expansion with partial quotients $n_1,n_2,\ldots$ The limit set of $S(E)$ is thus
\[
\Lambda_{S(E)} = F_E \df \{[0;\an_1,\an_2,\ldots] : \an_1,\an_2,\ldots \in E\}.
\]

\subsection{Examples of CIFS that are not PACIFS} Though this paper is concerned with PACIFSes, we include two non-examples for the benefit of our readers who are familiar with the well-studied notion of CIFSes.

\begin{example}
\label{examplecocantor}
Let $C$ be the middle-thirds Cantor set, and let $\II$ be the unique disjoint collection of intervals such that
\[
[0,1]\butnot C = \bigcup_{I\in \II} \Int(I).
\]
For each $I\in \II$, let $u_I:[0,1] \to I$ be the unique order-preserving bijective similarity between $[0,1]$ and $I$. Then $(u_I)_{I\in\II}$ is a similarity IFS (and thus also a conformal IFS), but it cannot be realized as a PACIFS. Indeed, if $(u_b)_{b\in S}$ is a PACIFS then $u_b \to p$ uniformly for some point $p$ (with $p=0$ according to our convention), but if $(I_n)$ is a sequence of distinct elements of $\II$, then the limit of the sequence $(u_{I_n})$ can be any point in $C$, and in particular is not limited to a single point.
\end{example}

\begin{example}
\label{exampleNotPACIFS}
For each $n$ let $u_n:[0,1] \to [0,1]$ be defined by
\[
u_n(x) = \frac{1 + x^n}{2^n}\cdot
\]
Then $(u_n)_{n\geq 1}$ is a conformal IFS, since the sequence $(u_n)$ is bounded in the $\CC^2$ norm. However, it cannot be obviously realized as a PACIFS, since this would require a finite-dimensional space $U$ to be able to parameterize the sequence $(u_n)$, and no such space exists. (The sequence $(u_n)$ is most obviously parameterized by the infinite-dimensional sequence $((e_0 + e_n)/2^n) \in \R^\N$ under the function $u(b,x) = \sum_k b_k x^k$, but moving to infinite dimensions would cause problems with the convergence of some series used the proof of Theorem \ref{maintheorem} below.)
\end{example}

\subsection{Symbolic and geometric Perron--Frobenius operators}

For the remainder of this section, we fix $(U,V,u,v,q)$ and let $(u_b)_{b\in S}$ be a PACIFS as in Definition \ref{definitionPACIFS}. 

\begin{definition}
\label{definitionsymbolicPF}
Fix $s \in \R$. If $\sum_{b\in S} |b|^{qs} < \infty$, we let $\w L = \w L_{S,s}:C(\Sigma)\to C(\Sigma)$ denote the \textbf{symbolic Perron--Frobenius operator}
\begin{equation}
\label{symbolicPF}
\w L f(\omega) \df \sum_{b\in S} \left|u_b'\circ\pi(\omega)\right|^s f(b\ast \omega),
\end{equation}
where $\ast$ denotes concatenation, and we let 
\[
P = P(S,s) \df \log\rho(\w L)
\] 
denote the logarithm of the spectral radius of $\w L$. If $\sum_{b\in S} |b|^{qs} = \infty$, then $\w L$ is not defined and we instead let $P \df +\infty$. Note that by \eqref{vdef}, we have 
\begin{equation}
\label{partitionfunction}
e^P \asymp \|\w L\| \asymp \sum_{b\in S} |b|^{qs},
\end{equation}
where the middle expression is interpreted as $\infty$ when $\w L$ is not defined.
Here the implied constants may depend on $(U,V,u,v,q)$ but not on the PACIFS $(u_b)_{b\in S}$.
\end{definition}

We now wish to recall several results from \cite{MauldinUrbanski2}. These results are generally stated for what \cite{MauldinUrbanski2} calls CIFSes, and what we will call OSC CIFSes (because they assume the open set condition in addition to conformality). Although PACIFSes are not necessarily OSC CIFSes, we can show that they satisfy \cite[\64.2: (4d),(4e)]{MauldinUrbanski2} in the definition of OSC CIFSes as well as parts of \cite[\64.2: (4a),(4c)]{MauldinUrbanski2}:
\begin{itemize}
\item[(4a),(4d)] Let $X = \{x\in V : \dist(x,\Lambda_S) \leq \epsilon\}$ for some sufficiently small $\epsilon > 0$. This satisfies all desired properties except connectedness.
\item[(4c)] Let $W = V$. This satisfies all desired properties except that the extension may not be globally invertible (it is locally invertible).
\item[(4e)] By \eqref{vdef}, this is true with $\alpha=1$.
\end{itemize}
These properties are enough to prove the following results for all PACIFSes. However, we note that for our main example of Gauss IFSes, all conditions of the OSC CIFS definition are satisfied.

\begin{itemize}
\item[(A1)] Convex and decreasing pressure function: $P(S,\cdot)$ is equal to the standard pressure function of $(u_b)_{b\in S}$ (cf.~\cite[(2.1) / pp.54-55 / p.78]{MauldinUrbanski2}) and in particular is convex and decreasing \cite[Proposition 4.2.8(b)]{MauldinUrbanski2}
\item[(A2)] Existence of eigenfunctions and eigenmeasures: for each $s\geq 0$ such that $-\infty < P < +\infty$, for some $\beta > 0$ there exist a positive $\beta$-H\"older continuous function $\w g = \w g_{S,s} \in \BB = \HH^\beta(\Sigma)$ and a positive measure $\w \mu = \w \mu_{S,s} \in \MM_+(\Sigma) \subset \BB^*$ such that 
\begin{align*}
&\w L \w g = e^P \w g &
&\text{and} &
&\w \mu \w L = e^P \w \mu
\end{align*}
see \cite[Theorem 2.7.3 / 3.2.3 / 6.1.2 and Theorem 2.4.3]{MauldinUrbanski2}. Here, we recall that $\w g$ is called $\beta$-H\"older continuous if $\sup_{x,y} \frac{\dist(\w g(x),\w g(y))}{\dist(x,y)^\beta} < \infty$. Note that if $P = 0$, then this means that $\w g$ and $\w \mu$ are right and left fixed points, respectively, for $\w L$.
\item[(A3)] Spectral gap: with $s,\beta$ as above, if $c = 1/\w \mu \w g$, then 
\[
\|e^{-n P} \w L^n - c \w g \w \mu\|_\beta \leq C \gamma^n
\]
for some $C < \infty$ and $\gamma < 1$ \cite[Theorem 2.4.6(b)]{MauldinUrbanski2}. Here the notation $\|\cdot\|_\beta$ means that the operator norm is taken with respect to the space $\BB = \HH^\beta(\Sigma)$, rather than the space $C(\Sigma)$ that $\w L$ was originally defined on.
\item[(A4)] Invariant measure: the shift map $\sigma: \Sigma \to \Sigma$ defined by
\[
\sigma(b \ast \omega) = \omega
\]
has the invariant measure
\[
\w\nu = c \w \mu M_{\w g},
\]
where $M_{\w g}$ denotes the operator of multiplication by $\w g$ \cite[Proposition 2.4.7]{MauldinUrbanski2}.
\end{itemize}

\begin{remark}
\label{remarkspectralgap}
In (A3), we mean a spectral gap in the sense of \cite{Rugh2}. Indeed, let $\w R \df \w L - c e^P \w g \w \mu$. The inequality $\|\w R^n\| \leq C e^{P n} \gamma^n$ implies that $\rho(\w R) \leq e^P \gamma < \rho(\w L)$, and conversely if $\rho(\w R) < \rho(\w L)$, we may take $\gamma \in e^{-P} (\rho(\w R),\rho(\w L))$, and then we have $\|\w R^n\| \leq C e^{n P} \gamma^n$ for some $C$. So the inequality $$e^{-n P} \|\w R^n\| = \| e^{-n P} \w L^n - c \w g \w \mu\| \leq C \gamma^n$$ is equivalent to the operator $\w L$ having a spectral gap in the sense of \cite{Rugh2}: $\w L$ has a simple isolated eigenvalue the modulus of which equals $\rho(\w L)$, and that the remaining part of the spectrum is contained in a disk centered at zero and of radius strictly smaller than $\rho(\w L)$.
\end{remark}

For the purpose of later calculation, we define and compute the Lyapunov exponent of the dynamical system $(\sigma,\w\nu)$:
\begin{equation}
\label{chidef}
\begin{split}
\chi &\df \int_\Sigma \log\left|(u_{\omega_1}^{-1})'\circ\pi(\omega)\right|\; \dee \w\nu(\omega)\\
&= \sum_{b\in S} \int_{b \ast \Sigma} \log\left|(u_b^{-1})'\circ\pi(\omega)\right| \; \w g(\omega) \;\dee\w \mu(\omega)\\
&= -e^{-P} \int_\Sigma \sum_{b\in S} \left|u_b'\circ\pi(\omega)\right|^s \log\left|u_b'\circ\pi(\omega)\right| \; \w g(b \ast \omega) \;\dee\w \mu(\omega)
\end{split}
\end{equation}
Note that $\chi > 0$, since $|(u_{\omega_1}^{-1})' \circ \pi(\omega)| > 1$ for all $\omega\in\Sigma$.

In what follows we will need a version of the Perron--Frobenius operator that operates on the space of holomorphic functions on a complex neighborhood of $\Lambda_S$. 

\begin{definition}
\label{definitioncomplexPF}
Let $U_\C,V_\C\subset \C$ be neighborhoods of $S\cup\{0\}$ and $\cl{\Lambda_S}$, respectively, such that $V_\C$ is connected, and $u,v$ can be extended to bounded holomorphic functions from $U_\C\times V_\C$ to $V_\C$ and to $\C$ respectively, such that the family of maps $(u_b)_{b\in U}$ is still uniformly contracting.
For each $s \in \R$ such that $\sum_{b\in S} |b|^{qs} < \infty$ we consider the \textbf{geometric Perron--Frobenius operator} $L = L_{S,s}: C(V_\C)\to C(V_\C)$ defined by the formula
\begin{equation}
\label{complexPF}
L f(x) \df \sum_{b\in S} |b|^{q s} e^{s v(b,x)} \; f\circ u_b(x),
\end{equation}
which we will consider throughout the paper. Note that $L$ is related to the $\w L$ by the semiconjugacy relation 
\[
\w L \,\Pi = \Pi \,L,
\] 
where the operator $\Pi : \Lip(V_\C) \to \HH^\beta(\Sigma)$ defined by $\Pi f = f\circ \pi$ is continuous but usually not surjective. 
\end{definition}

Note that the following bounded distortion property holds: for all $n\in\N$, $\omega \in S^n$, and $x \in V_\C$, we have 
\[ \big|v(\omega_1,x) + v(\omega_2,u_{\omega_1}(x)) + \ldots + v(\omega_n,u_{\omega_{n-1}}\circ\cdots\circ u_{\omega_1}(x)) \big| \lesssim 1. 
\] 
This is because of the uniform contraction property of $(u_b)_{b\in S}$, together with the Lipschitz continuity of $v$.

The arguments of \cite[\62]{MauldinUrbanski2} can be easily adapted to the setting of \eqref{complexPF}, yielding the following results:
\begin{itemize}
\item[(B1)] Spectral Radius: Due to the bounded distortion property described above, the spectral radii of the operators \eqref{symbolicPF} and \eqref{complexPF} are both equal to $e^P$.
\item[(B2)] Existence of eigenfunctions and eigenmeasures: If $-\infty < P < +\infty$, then there exist a Lipschitz continuous function $g \in \BB = \Lip(V_\C)$ which is positive on $V_\R \df V_\C \cap \R$ and a positive measure $\mu \in \MM_+(V_\C) \subset \BB^*$ such that 
\begin{align*}
&L g = e^P g &
&\text{and} &
&\mu L = e^P \mu.
\end{align*} 
After renormalization, we have $\Pi \,g = \w g$ and $\w\mu \,\Pi = \mu$; in particular, $\mu$ is supported on $\Lambda_S$.
\item[(B3)] Spectral gap\Footnote{See Remark \ref{remarkspectralgap}.}: We have 
\[
\|e^{-n P} L^n - c g \mu\|_1 \leq C \gamma^n
\] 
for some $C < \infty$ and $\gamma < 1$, where $c = 1/\mu g$, and $\|\cdot\|_1$ indicates that the operator norm is being taken with respect to the space $\BB = \Lip(V_\C)$, rather than the space $C(V_\C)$ that $L$ was originally defined on.
\item[(B4)] Lyapunov exponent: Using the formulas $\w L \,\Pi = \Pi \,L$, $\Pi \,g = \w g$, and $\w \mu\, \Pi = \mu$, we get
\[
\chi = -c e^{-P} \mu \alpha_1 g,
\]
where
\begin{equation}
\label{alpha1def}
\alpha_1 f(x) \df \sum_{b\in S} |u_b'(x)|^s \log|u_b'(x)| \; f\circ u_b(x).
\end{equation}
In what follows we will need to consider the ``unnormalized'' Lyapunov exponent
\[
\w\chi \df c^{-1} e^P \chi = -\mu \alpha_1 g > 0.
\]
\end{itemize}
Now by (B3) we have
\[
e^{-n P} L^n \one \to c g \mu \one
\]
uniformly, and since $L$ preserves the space of holomorphic functions, it follows that $g$ is holomorphic\Footnote{A similar result was proven in \cite[Corollary 6.1.4]{MauldinUrbanski2}, though the hypotheses and conclusion are somewhat different. Note that the invariance hypothesis on $U$ in \cite[Corollary 6.1.4]{MauldinUrbanski2} should be that each element of $S$ can be extended to a univalent holomorphic map from $U$ to itself, rather than what is written there.}.

Although (B3) is stated only for the Lipschitz norm $\|\cdot\|_1$, for holomorphic functions it holds for the sup norm $\|\cdot\|_\infty$ as well. Indeed, recall Cauchy's inequality: for every bounded holomorphic function $f$ whose domain includes $B_\C(z,\rho)$ we have
\begin{equation}
\label{cauchy}
\frac1{i!} |f^{(i)}(z)| \leq \rho^{-i} \|f\|_\infty .
\end{equation}
In particular, if $K \subset V_\C$ is compact then $\|f \given K\|_{1} \lesssim \|f\|_\infty$. Since $\bigcup_{b\in U} u_b(V_\C)$ is precompact in $V_\C$, it follows that $\|L f\|_1 \lesssim \|f\|_\infty$ and thus
\[
\|e^{-n P} L^n f - c g \mu f\|_\infty \leq \|e^{-n P} L^n f - c g \mu f\|_1 \leq C \gamma^{n - 1} \|L f\|_1 \lesssim \gamma^n \|f\|_\infty
\]
and therefore we have 
\[
\|e^{-n P} L^n - c g \mu\|_\infty \lesssim \gamma^n.
\]

The value of $s$ such that $P(S,s) = 0$ is particularly important, if such a value exists, hence we make the following definition:
\begin{definition}[Cf. {\cite[p.78 and Definition 4.3.1]{MauldinUrbanski2}}]
Given a set $S\subset U\butnot\{0\}$ as in Definition \ref{definitionPACIFS}, we call $S$ as well as the associated PACIFS $(u_b)_{b\in S}$ \emph{regular} if there exists $\delta\geq 0$ such that $P(S,\delta) = 0$, and \emph{strongly regular} if furthermore there exists $\kappa > 0$ such that $P(S,\delta - \kappa) < +\infty$. Equivalently, $S$ is strongly regular if there exists $s$ such that $0 < P(S,s) < +\infty$.
\end{definition}

\subsection{Bowen's formula}

In what follows we let
\[
\delta = \delta_S \df \inf\{s \in \R : P(S,s) \leq 0\},
\]
and we notice that $P(S,\delta_S) = 0$ if and only if $S$ is regular. If $S$ is regular, we write $L_S = L_{S,\delta_S}$. We also let
\[
\Theta_S \df \inf\{s \in\R : P(S,s) < +\infty\}
\]
and we note that $S$ is strongly regular if and only if $\delta_S > \Theta_S$.

Finally, for our last result we need to assume the OSC, which we define as follows:
\begin{definition}
\label{definitionOSCPACIFS}
A PACIFS $(u_b)_{b\in S}$ satisfies the \emph{open set condition (OSC)} if there exists a connected open set $W$ precompact in $V$ such that:
\begin{itemize}
\item $(u_b(W))_{b\in S}$ is a disjoint collection of subsets of $W$;
\item for each $b \in S$, $u_b$ is injective.
\end{itemize}
\end{definition}
It is easily verified (by letting $X = \cl W$) that every OSC PACIFS is an OSC CIFS (as recalled in Appendix \ref{appendixCIFS}).

For every OSC PACIFS, we have the following:
\begin{itemize}
\item[(A5,B5)] Bowen's formula: The Hausdorff dimension of $\Lambda_S$ is 
\[
\HD(\Lambda_S) = \delta_S,
\] 
see \cite[Theorem 4.2.13]{MauldinUrbanski2}.
Note that the pressure function $P(S,\cdot)$ appearing in the definition of $\delta_S$ can be expressed as either $P(S,s) = \log \rho(\w L)$ or as $P(S,s) = \log\rho(L)$, so this result can be thought of as being about both the symbolic and the geometric Perron--Frobenius operators.
\end{itemize}


\section{Statement of main result}

\label{sectiontheoremscase12}

Our main result is an application of Theorem \ref{theoremoperatorequation} to the situation where a PACIFS $(u_b)_{b\in S}$ is being approximated by another PACIFS $(u_b)_{b\in S'}$. Thus, we fix $(U,V,u,v,q)$, and we let $U_\C,V_\C\subset \C$ be as above.

In what follows all operators will be interpreted as acting on the Banach space $\BB = \H(V_\C)$, where $\H(V_\C)$ denotes the Banach space of bounded holomorphic functions on $V_\C$, endowed with the sup norm.

\ignore{
\begin{theorem}
\label{theoremcase1}
Let $(U,V,u,v,q)$ be as in Definition \ref{definitionPACIFS}, and fix $S \subset U\butnot\{0\}$. Let $\delta = \delta_S$, fix $\kappa > 0$, and suppose that $\zeta(\kappa) < \infty$. Then there exist $\zetabound,\epsilon > 0$ such that for all $S' \subset U\butnot\{0\}$ such that $\zeta(\kappa) \leq \zetabound$ and $\|S\triangle S'\| \leq \epsilon$, we have
\begin{equation}
\label{case1formula}
\delta_{S'} = \delta + \sum_{\substack{I\\ \mj(I)\leq \#(I) - 1}} c_I \, \eta_I
\end{equation}
for some constants $c_I$, with $c_\0 = 0$ and $c_{\{(0,0)\}} = \frac{\mu h \nu g}{\w\chi} > 0$, where $\0$ is the empty multiset and $\{(0,0)\}$ is the singleton multiset containing $(0,0)$ with multiplicity 1. Moreover, for all $I$ we have
\[
|c_I| \lesssim \zetabound^{-\#(I)} \epsilon^{-\Sigma(I)} \kappa^{-\mj(I)}
\]
and thus for all $K,I$,
\[
\delta_{S'} = \delta + \sum_{\substack{I \\ \#(I) < K \\ \Sigma(I) \leq I}} c_I \, \eta_I(S') = O\left(\left(\frac{\zeta(\kappa)}{\zetabound}\right)^K \left(\frac{\|S\triangle S'\|}{\epsilon}\right)^I\right).
\]
\end{theorem}

This theorem includes Scenario 1 as shown by the following corollaries:

\begin{corollary*}
\label{corollarycase1a}
Let $(S_N)$ be an ascending sequence of sets converging to a strongly regular set $S$. Then \eqref{case1formula} holds for all $N$ sufficiently large.
\end{corollary*}
\begin{corollary*}
\label{corollarycase1b}
Let $(S_N)$ be a descending sequence of sets converging to a set $S$, and suppose that $P(S_M,\delta - \kappa) < +\infty$ for some $M\in\N$ and $\kappa > 0$. Then \eqref{case1formula} holds for all $N$ sufficiently large.
\end{corollary*}
\begin{proof}[Proof of both corollaries]
Let $S_* = S$ for the first corollary and $S_* = S_M$ for the second corollary. Either way, there exists $\kappa > 0$ such that $P(S_*,\delta - \kappa) < +\infty$. Then $\sum_{b\in S_*} |b|^{q(\delta - \kappa)} < \infty$, so $\zeta(\kappa) < \infty$, and by the dominated convergence theorem $\sum_{b\in S\triangle S_N} |b|^{q(\delta - \kappa)} \to 0$ as $N\to\infty$. Thus for all sufficiently large $N$, $\zeta(\kappa) \leq \zetabound$ and thus Theorem \ref{theoremcase1} applies.
\end{proof}

\begin{remark*}
Note that if $\delta = 0$ and $S_M$ is infinite for all $M$, then $P(S_M,\delta - \kappa) \geq P(S_M,0) = \log\#(S_M) = \infty$, so Corollary \ref{corollarycase1b} does not apply. More generally, if $\delta = 0$ then $\zeta(\kappa) \geq 1$ for all infinite $S'$ and $\kappa > 0$, so Theorem \ref{theoremcase1} cannot be used to compute $\delta_{S'}$ for such $S$.
\end{remark*}

Now, Theorem \ref{theoremcase1} cannot be used to prove Theorem \ref{theoremFgeqN} because the inequality $\zeta(S_N,\kappa) \leq \zetabound$ is not satisfied. Our next theorem will generalize Theorem \ref{theoremcase1}, in a way that includes the case of Theorem \ref{theoremFgeqN}.

}

\begin{notation*}
Fix $S,S' \subset U\butnot \{0\}$ and $\delta\in\R$ such that $\delta \geq \delta_S$, and let
\[
\theta = \theta(S') \df \delta_{S'} - \delta.
\]
For each $i\geq 0$, we let
\begin{equation}
\label{etaidef}
\eta_i = \eta_i(S') \df \sum_{b \in S'} |b|^{q(\delta + \theta)} b^i - \sum_{b \in S} |b|^{q(\delta + \theta)} b^i.
\end{equation}
Note that $\eta_0$ is positive when $S'\supset S$ and negative when $S'\subset S$, as long as $S' \neq S$. Next, we define
\[
\etaratio \df \sup_{b\in S\triangle S'} |b|,
\]
where $S \triangle S' \df (S\butnot S') \cup (S'\butnot S)$ is the symmetric difference of $S$ and $S'$.

Let $\MM(A)$ denote the set of all multisets on a set $A$, i.e. finitely supported functions from $A$ to $\N$. If $I\in\MM(A)$, then $I(i) = n$ is interpreted as meaning ``$i$ is an element of $I$ of multiplicity $n$''. We denote the empty multiset by $\0$, and for each $i\in A$, we denote the singleton multiset containing $i$ by $\{i\}$, so that $\{i\}(i) = 1$, and $\{i\}(i') = 0$ for $i' \neq i$. Note that this implies that e.g. $i\{j\}$ denotes the multiset containing $j$ with a multiplicity of $i$.

For $I\in\MM(\N)$, we write
\begin{align*}
\#(I) &\df \sum_i I(i), &
\Sigma(I) &\df \sum_i i I(i), &
\eta_I &\df \prod_i \eta_i^{I(i)},
\end{align*}
where the summations and product are taken over the finite set $\{i\in\N : I(i) > 0\}$.
Finally, let
\begin{align*}
h(x) &\df e^{\delta v(0,x)}, &
\nu f &\df f(0), &
L_1 &\df L_{S,\delta},
\end{align*}
\begin{align*}
\hspace{58pt} \cbar\; &\df 1/\sum_{m = 0}^\infty \nu L_1^m h, &
\xi = \xi(S') &\df \eta_0 - \cbar.
\hspace{36pt}
\end{align*}
\end{notation*}

\begin{remark*}
Since $\delta \geq \delta_S$, we have $P(S,\delta) \leq 0$, or equivalently $\rho(L_1) = e^{P(S,\delta)} \leq 1$, where as before $\rho$ denotes spectral radius. It follows that $\cbar = 0$ if and only if $P(S,\delta) = 0$. Since $\delta \geq \delta_S$, it follows that $\cbar = 0$ if and only if both (a) $\delta = \delta_S$ and (b) $S$ is regular.
\end{remark*}

\begin{theorem}
\label{maintheorem}
With notation as above, fix $S \subset U\butnot\{0\}$ and $\delta\in\R$ such that $\delta \geq \delta_S$ and $\delta > \Theta_S$. Then there exist $\epsilon > 0$ and explicitly computable constants $c_{I,j,k}$ with $c_{\0,0,0} = 0$ and $c_{\0,0,1} = 1$ such that for all regular $S'\subset U\butnot\{0\}$ satisfying $\etaratio,|\theta| \leq \epsilon$, 
we have 
\begin{equation}
\label{mainformula}
\Xi = 0,
\;\;\; \text{ where } \;\;\;\;\;\;
\Xi = \Xi(S') \df \sum_{I\in \MM(\N_{\geq 1})} \sum_{j = 0}^\infty \sum_{k = 0}^\infty c_{I,j,k} \, \eta_I \theta^j \xi^k.
\end{equation}
If $\cbar = 0$, then $\xi = \eta_0$ and thus the right half of \eqref{mainformula} can be rewritten as
\begin{equation}
\label{mainformula2}
\Xi = \sum_{I\in \MM(\N)} \sum_{j = 0}^\infty c_{I,j} \, \eta_I \theta^j
\end{equation}
where $c_{I,j} = c_{(I\given\N_{\geq 1}),j,I(0)}$. In this case we have $c_{\0,1} = c_{\0,1,0} = -\w\chi < 0$, where $\w\chi$ is as in \text{(B4)}, where $g,\mu$ are right and left fixed points of $L_1$ normalized so that
\begin{equation}
\label{normalization}
\mu h = \nu g = 1.
\end{equation}
Moreover,
\begin{equation}
\label{cijkbounds}
|c_{I,j,k}| \lesssim \epsilon^{-(\Sigma(I) + j + k)}.
\end{equation}
When $\cbar = 0$, this can be written as
\begin{equation}
\label{cijkboundsv2}
|c_{I,j}| \lesssim \epsilon^{-(\Sigma(I) + I(0) + j)}.
\end{equation}
\end{theorem}

\begin{remark*}
If we assume $\delta \geq \delta_S$, then the hypothesis that $\delta > \Theta_S$ is satisfied if and only if either (a) $\delta_S > \Theta_S$ (i.e. $S$ is strongly regular) or (b) $\delta > \delta_S$.
\end{remark*}

\begin{remark*}
It is natural to let $\delta = \lim_{N\to \infty} \delta_{S_N}$, where $S_N \to S$ is a sequence such that this limit exists, from which $S'$ will be chosen. In this case, we automatically have $|\theta| \leq \epsilon$ for all $N$ sufficiently large, and $\delta \geq \delta_S$ automatically due to semicontinuity of Hausdorff dimension for CIFS limit sets \cite[Theorem 4.2.13]{MauldinUrbanski2}. Moreover, if $S = \emptyset$ but $S_N \neq \emptyset$, then $\delta \geq 0 > -\infty = \delta_S$, so the hypothesis $\delta > \Theta_S$ is satisfied despite the fact that $S = \emptyset$ is not strongly regular (and in fact is not regular at all).
\end{remark*}

\begin{corollary}
\label{corollaryanalytic}
Fix $S,\delta$ as in Theorem \ref{maintheorem}, and let $(S_N)$ be a sequence of sets. Suppose that for some $d\in\N$, there exist a sequence $(\tt_N)$ in $\C^d$ converging to $\0$, and functions $F_i,F_\ast \in \H(B)$ holomorphic on a fixed neighborhood $B$ of $(\0,0)\in \C^{d+1}$, such that for each $N$,
\[
\eta_i(S_N) = F_i(\tt_N,\theta_N), \;\;\;\;
\theta_N = \theta(S_N),  \;\;\;\;
\xi(S_N) = F_*(\tt_N,\theta_N).
\]
Furthermore, suppose that 
\begin{align*}
&\|F_i\| \leq \epsilon^i/2^i, & 
&\|\pi_2 \| \leq \epsilon/2, \;\text{and}\; &
&\|F_*\| \leq \epsilon/2, 
\end{align*}
where $\pi_2(\tt,\theta) = \theta$ is the projection onto the second coordinate, and $\epsilon$ is as in Theorem \ref{maintheorem}. Then
\[
\Xi(S_N) = F(\tt_N,\theta_N)
\]
where $F$ is a holomorphic function defined on $B$.
\end{corollary}
\begin{proof}
Define the function 
\[
F \;\df \sum_{I\in \MM(\N_{\geq 1})} \sum_{j = 0}^\infty \sum_{k = 0}^\infty c_{I,j,k} \left(\prod_{i\in I} F_i\right) \pi_2^j F_*^k 
\;=\; \sum_{I,j,k} c_{I,j,k} \left(\prod_{i\in I} F_i\right) \pi_2^j F_*^k
\]
Note that the bound \eqref{cijkbounds} guarantees that the above series converges absolutely, since 
\begin{align*}
\sum_{I,j,k} \left\| c_{I,j,k} \left(\prod_{i\in I} F_i\right) \pi_2^j F_*^k \right\|
&\underset{\eqref{cijkbounds}}{\lesssim} \sum_{I,j,k} \epsilon^{-(\Sigma(I) + j + k)} (\epsilon/2)^{\Sigma(I) + j + k}\\
&= \sum_{I,j,k} (1/2)^{\Sigma(I) + j + k}\\
&= \left(\prod_{i = 1}^\infty \sum_{\ell = 0}^\infty (1/2)^{i\ell}\right)
\left(\sum_{j = 0}^\infty (1/2)^j\right)
\left(\sum_{k = 0}^\infty (1/2)^j\right)\\
&= 4 \prod_{i = 1}^\infty (1 - 2^{-i})^{-1}
< \infty
\end{align*}
Therefore the series defining $F$ converges in $\H(B)$.
By the definition of $F$, we have 
\[
F(\tt_N,\theta_N) \underset{\eqref{mainformula}}{=} \Xi(S_N).
\] 
Indeed,
\begin{align*}
F(\tt_N,\theta_N) &= \sum_{I,j,k} c_{I,j,k} \left(\prod_{i\in I} F_i (\tt_N,\theta_N) \right)  \cdot \pi_2^j (\tt_N,\theta_N) \cdot F_*^k (\tt_N,\theta_N)\\
&= \sum_{I,j,k} c_{I,j,k} \left(\prod_{i\in I} \eta_i(S_N) \right) \theta(S_N)^j \xi(S_N)^k \underset{\eqref{mainformula}}{=} \Xi(S_N)
\qedhere
\end{align*}
\end{proof}

\section{Spectral gap in the case $\cbar > 0$}
\label{sectionspectralgap}

To prove Theorem \ref{maintheorem}, we need to apply Theorem \ref{theoremoperatorequation}; thus, given sets $S,S'$, we need to produce operators $L,L'$, satisfying the hypotheses of Theorem \ref{theoremoperatorequation} if $S'$ is a sufficiently close perturbation of $S$, i.e. one for which $\etaratio,|\theta| \leq \epsilon$ as in Theorem \ref{maintheorem}. When $\cbar = 0$ we can take $L = L_S$ and $L' = L_{S'}$, since we have $P(S,\delta) = 0$ and thus by (B2) of \6\ref{sectionPACIFS}, $L_S$ has right and left fixed points. However, if $\cbar > 0$ then $P(S,\delta) < 0$ and thus $L_1 = L_{S,\delta}$ has spectral radius $<1$ and has neither right nor left fixed points. In this section we prove that there is another operator with right and left fixed points, which will be suitable to plug in for $L$ in Theorem \ref{theoremoperatorequation}. Moreover we prove that this operator has a spectral gap, guaranteeing that the series $\sum_n \|R^n\|$ appearing in Theorem \ref{theoremoperatorequation} converges.

\begin{proposition}
\label{propositioncbarspectralgap}
Let $L_1$ be an operator on a Banach space $\BB$ such that $\rho(L_1) < 1$, where $\rho$ denotes spectral radius. Fix $h\in\BB$, $\nu\in\BB^*$ such that $\nu h > 0$, and $\nu L_1^m h \geq 0$ for all $m \geq 0$. Then if we let
\begin{align*}
Q_1 &\df \sum_{m = 0}^\infty L_1^m, &
\cbar &\;\df 1/\nu Q_1 h, &
L &\df L_1 + \cbar \, h \nu,
\end{align*}
\begin{align*}
g &\df Q_1 h, &
\mu &\df \nu Q_1, 
\end{align*}
then it follows that 
\[
L g = g 
\;\text{and}\; 
\mu L = \mu.
\]
Moreover, there exist $C < \infty$ and $\gamma < 1$ such that for all $n$,
\begin{equation}
\label{equationspectralgap}
\|L^n - c g \mu\| \leq C \gamma^n
\end{equation}
where $c \df 1/\mu g$. In particular, $L$ has a spectral gap in the sense of \cite{Rugh2}\Footnote{See Remark \ref{remarkspectralgap}.}.
\end{proposition}
\begin{proof}
We have
\[
L g = (L_1 + \cbar \, h \nu) Q_1 h = \sum_{m = 1}^\infty L_1^m h + \cbar \, h \nu Q_1 h = \sum_{m = 0}^\infty L_1^m h = g
\]
and similarly $\mu L = \mu$. Let $a_{m + 1} = \cbar \, \nu L_1^m h$ and $b_{n + 1} = \cbar \, \nu L^n h$ for $m,n\geq 0$, and let $a_0 = 0$ and $b_0 = 1$. Then since $L^n = (L_1 + \cbar \, h \nu)^n$ is the sum of all $n$-fold ordered products of $L_1$ and $\cbar \, h \nu$, we have
\begin{align*}
b_{n + 1} = \cbar \, \nu L^n h
&= \cbar \, \nu (L_1 + \cbar \, h \nu)^n h\\
&= \sum_{\kr = 1}^\infty \sum_{\substack{m_1,\ldots,m_\kr \\ \sum_i (m_i + 1) - 1 = n}} \cbar \, \nu L_1^{m_1} (\cbar \, h\nu) L_1^{m_2} \cdots L_1^{m_{\kr - 1}} (\cbar \, h \nu) L_1^{m_\kr} h\\
&= \sum_{\kr = 1}^\infty \sum_{\substack{m_1,\ldots,m_\kr \\ \sum_i m_i = n - \kr + 1}} \prod_{i = 1}^\kr \cbar \, \nu L_1^{m_i} h\\
&= \sum_{\kr = 1}^\infty \sum_{\substack{m_1,\ldots,m_\kr \\ \sum_i m_i = n + 1}} \prod_{i = 1}^\kr a_{m_i}.
\end{align*}
Together with the equality $b_0 = 1$, this shows that the sequence $(b_n)$ is the sum of the $\kr$-fold convolutions of the sequence $(a_m)$ over $\kr\in\N$.

Now let $A$ and $B$ be the functions whose Taylor series coefficients are given by $(a_m)$ and $(b_n)$, i.e.
\begin{align*}
A(z) &= \sum_{m = 0}^\infty a_m z^m, &
B(z) &= \sum_{n = 0}^\infty b_n z^n.
\end{align*}
Then the convolution relation between $(a_m)$ and $(b_n)$ mentioned above implies that
\[
B(z) = \sum_{\kr = 0}^\infty [A(z)]^\kr = \frac{1}{1 - A(z)}
\]
for all $z$ in the radius of convergence of both $A$ and $B$. Now by hypothesis we have $\lambda \df \rho(L_1) < 1$, and by definition of the spectral radius we have $|a_m| \lesssim_m (\lambda + \epsilon)^m$ for all $\epsilon > 0$. Thus the series defining $A$ converges in the open ball $B_\C^\circ(0,\lambda^{-1}) \supset B_\C(0,1)$. Moreover, by the definition of $\cbar$ we have $f(1) = \sum_m a_m = 1$, and by hypothesis we have $a_1 > 0$ and $a_m \geq 0$ for all $m$. We claim that for all $z\in B_\C(0,1)$, if $A(z) = 1$ then $z = 1$. Indeed, since $|a_m z^m| \leq a_m$ and $\sum_m a_m = 1$, if $A(z) = 1$ then we must have $a_m z^m = a_m$ for all $n$. In particular $a_1 z = a_1$, and since $a_1 > 0$ this implies $z = 1$.

Next, we observe that
\[
A'(1) = \sum_{m = 0}^\infty m a_m = \frac1r \df \frac{\cbar}{c} = \frac{\mu g}{\nu Q_1 h} > 0.
\]
It follows that $B$ can be extended to a meromorphic function $\what B = \frac{1}{1 - A}$ on $B_\C^\circ(0,\lambda^{-1})$, and the only pole of $\what B$ in $B_\C(0,1)$ is $1$, where $\what B$ has a simple pole of residue $-r$. So
\[
\what B(z) = \frac{r}{1 - z} + E(z)
\]
where $E$ is a meromorphic function on $B_\C^\circ(0,\lambda^{-1})$, which is holomorphic on a closed neighborhood of $B_\C(0,1)$, say $B_\C(0,\tau^{-1})$ with $\lambda < \tau < 1$. Since $\frac{r}{1 - z} = \sum_n r z^n$, it follows from Cauchy's inequality (cf. \eqref{cauchy}) that
\[
|b_n - r| \leq \tau^n \| E \|_\infty.
\]
Now
\begin{align*}
L^n = (L_1 + \cbar \, h \nu)^n &= L_1^n + \sum_{i = 0}^{n - 1} L_1^i (\cbar \, h \nu) L_1^{n - i - 1} + \sum_{i = 0}^{n - 2} \sum_{j = 0}^{n - i - 2} L_1^i (\cbar \, h \nu) L^{n - i - j - 2} (\cbar \, h \nu) L_1^j\\
&= L_1^n + \sum_{i = 0}^{n - 1} \sum_{j = 0}^{n - i - 1} \cbar \, L_1^i h b_{n - i - j - 1} \nu L_1^j
\end{align*}
and thus after setting $b_n = 0$ when $n < 0$ and $\|b\| = \sup_n |b_n|$, we have
\begin{align*}
\|L^n - c g \mu\| &= \|L^n - r \, \cbar \, g \mu\|\\
&= \left\|L_1^n + \sum_{i = 0}^\infty \sum_{j = 0}^\infty \cbar(b_{n - i - j - 1} - r) L_1^i h \nu L_1^j\right\|\\
&\leq \|L_1^n\| + \sum_{i = 0}^\infty \sum_{j = 0}^\infty \cbar \, \min\big(r + \|b\|,|b_{n - i - j - 1} - r|\big) \|L_1^i\| \cdot \|h \nu\| \cdot \|L_1^j\|\\
&\lesssim \tau^n + \sum_{i = 0}^\infty \sum_{j = 0}^\infty \min(1,\tau^{n - i - j}) \tau^{i + j}\\
&\leq \tau^n + \sum_{i = 0}^\infty \sum_{j = 0}^\infty \min(1,\tau^{i - n}) \min(1,\tau^{j - n}) \tau^n\\
&= \tau^n + \left(\sum_{i = -n}^\infty \min(1,\tau^i)\right)^2 \tau^n \asymp n^2 \tau^n.
\end{align*}
Thus for any choice of $\gamma$ in $(\tau,1)$, we have that \eqref{equationspectralgap} is satisfied. This completes the proof of Proposition \ref{propositioncbarspectralgap}.
\end{proof}

\section{Proof of Theorem \ref{maintheorem}}
\label{sectionproof}


As before, recall that $\BB$ denotes the Banach space of bounded holomorphic functions on $V_\C$, endowed with the sup norm. All operator norms will be taken with respect to $\BB$.

Let $S,\delta$ be as in Theorem \ref{maintheorem}. Since $\delta > \Theta_S$, there exists $\kappa > 0$ such that $\delta - \kappa > \Theta_S$. Fix $\epsilon > 0$ to be determined. Now fix $S'\subset U\butnot\{0\}$ such that $\etaratio,|\theta| \leq \epsilon$, let
\begin{align*}
L &\df L_1 + \cbar \, h\nu, &
L' &\df L_{S',\delta + \theta}, &
L'' &\df L_{S,\delta + \theta} + \cbar \, h\nu,\\
\alpha &\df L'' - L, & 
\beta &\df L' - L'', &
\Delta &\df L' - L = \alpha + \beta.
\end{align*}
Since $S'$ is regular and $\delta + \theta = \delta_{S'}$, we have $P(S',\delta + \theta) = 0$ and thus $L' = L_{S'}$ has positive right and left fixed points $g'$ and $\mu'$ by (B2) of \6\ref{sectionPACIFS}. Here we call a function \emph{positive} if it is uniformly positive on $V_\R$, and we call an element of $\BB^*$ positive if it arises from a nonzero nonnegative measure supported on $V_\R$. In particular, if $f \in \BB$ and $\sigma \in \BB^*$ are both positive then $\sigma f > 0$.

If $\cbar = 0$, then $P(S,\delta) = 0$ and thus $L = L_S$ has positive right and left fixed points $g$ and $\mu$ by (B2), and by (B3), there exist constants $C < \infty$ and $\gamma < 1$ such that $\|R^n\| = \|L^n - c g \mu\| \leq C \gamma^n$ for all $n$, where $R$ is as in Theorem \ref{theoremoperatorequation}. On the other hand, if $\cbar > 0$, then the existence of such $g,\mu,C,\gamma$ follows from Proposition \ref{propositioncbarspectralgap}. Either way, we get $\sum_{n=0}^\infty \|R^n\| \leq C/(1-\gamma) < \infty$, so if
\begin{equation}
\label{Deltasufficient}
\|\Delta\| < (1-\gamma)/C,
\end{equation}
then the hypotheses of Theorem \ref{theoremoperatorequation} are satisfied, and consequently \eqref{operatorequation} holds. We aim to show that \eqref{Deltasufficient} holds if $\epsilon$ is sufficiently small, while simultaneously developing the tools that will allow us to reduce \eqref{operatorequation} to the equation $\Xi = 0$, where $\Xi$ is as in \eqref{mainformula}.

\begin{remark*}
In the remainder of the paper we will assume that $g,\mu$ are normalized as in \eqref{normalization}.
\end{remark*}

In what follows, the implied constants of asymptotics may depend on $S$ but not on $S'$ or $\theta$.

\begin{lemma}
\label{lemmaetabound}
If $\epsilon$ is sufficiently small, then for all $i$,
\[
|\eta_i| \lesssim \etaratio^i,
\]
where $\eta_i$ is as in \eqref{etaidef}.
\end{lemma}
\begin{proof}
If $\epsilon \leq \kappa$, then
\begin{align*}
|\eta_i| &\leq \sum_{b \in S\triangle S'} |b|^{q(\delta + \theta) + i} \leq \etaratio^i \sum_{b \in S\triangle S'} |b|^{q(\delta + \theta)}\\
&\leq \etaratio^i \left(\sum_{b\in S} |b|^{q(\delta + \theta)} + \sum_{b\in S'} |b|^{q(\delta + \theta)} \right)\\
&\lesssim \etaratio^i \left(\sum_{b\in S} |b|^{q(\delta - \kappa)} + \sum_{b\in S'} |b|^{q(\delta + \theta)} \right)\since{$|\theta| \leq \epsilon \leq \kappa$} \\
&\asymp \etaratio^i \left( e^{P(S,\delta - \kappa)} + e^{P(S',\delta + \theta)} \right) \by{\eqref{partitionfunction}} \\ 
&\asymp \etaratio^i
\end{align*}
where the last asymptotic is true since $P(S,\delta - \kappa) < +\infty$ and $P(S',\delta + \theta) = 0$, the former being true since we chose $\kappa$ such that $\delta - \kappa > \Theta_S$ and the latter since  $S'$ is regular and $\delta + \theta = \delta_{S'}$.
\end{proof}

\begin{lemma}
\label{lemmaalphaf}
If $\epsilon$ is sufficiently small, then
\begin{equation}
\label{alphaseries}
\alpha = \sum_{j\geq 1} \theta^j \alpha_j,
\end{equation}
where $\alpha_j \in \LL(\BB)$ is the unique operator such that
\begin{equation}
\label{alphajdef}
\alpha_j f(x) = \frac{1}{j!} \sum_{b\in S} |u_b'(x)|^\delta \log^j|u_b'(x)| \; f\circ u_b(x) \;\;\;\; \all f\in\BB \all x\in V_\R.
\end{equation}
The value of $\alpha_j f(x)$ for $x\in V_\C$ is obtained by replacing $|u_b'(x)|$ by $|b|^q e^{v(b,x)}$ in the above formula. Furthermore,
\begin{equation}
\label{alphabound}
\|\alpha_j\| \lesssim \kappa^{-j}.
\end{equation}
\end{lemma}
Note that $\alpha_0 = L_1$, and that $\alpha_1$ is as defined in \eqref{alpha1def}, with $s = \delta$.
\begin{proof}
Indeed, fix $f\in \BB$. For $x\in V_\R$ we have
\begin{align*}
\alpha f(x) &= \sum_{b\in S} |u_b'(x)|^\delta \big(|u_b'(x)|^\theta - 1\big) \; f\circ u_b(x)\\
&= \sum_{j=1}^\infty \frac{\theta^j}{j!} \sum_{b\in S} |u_b'(x)|^\delta \log^j|u_b'(x)| \; f\circ u_b(x)
\note{see below to justify interchange}\\
&= \sum_{j=1}^\infty \theta^j \alpha_j f(x).
\end{align*}
The assertion about the value of $\alpha_j f(x)$ for $x\in V_\C$ can be obtained either by analytic continuation, or by repeating the above calculation with the suggested substitution.

To demonstrate \eqref{alphabound}, we note that for all $x\in V_\C$,
\begin{align*}
|\alpha_j f(x)|
&\leq \frac{1}{j!} \sum_{b\in S} |b|^{q\delta} e^{\delta \Re v(b,x)} \big|q\log|b| + v(b,x)\big|^j \cdot |f\circ u_b(x)|\\
&\leq  \frac{1}{j!} \sum_{b\in S} |b|^{q\delta} e^{\delta \|v\|}  (-q\log|b| + C)^j \|f\|,
\end{align*}
where $C = \|v\| + 2q(\log\|U\|)_+$. Here $\|U\| = \sup_{b\in U} |b|$, and $(\cdot)_+$ denotes the positive part.

Let $w(b) = -q\log|b| + C \geq 0$. Then
\begin{align*}
\|\alpha_j\| &\leq \frac1{j!} \sum_{b\in S} e^{-\delta w(b)} w^j(b)
= \kappa^{-j} \sum_{b\in S} e^{-\delta w(b)} \frac{(\kappa w(b))^j}{j!}\\
&\leq \kappa^{-j} \sum_{b\in S} e^{-(\delta - \kappa) w(b)} \asymp \kappa^{-j} e^{P(S,\delta - \kappa)} \lesssim \kappa^{-j}
\end{align*}
since $P(S,\delta - \kappa) < +\infty$. Note that this calculation also shows that the interchange of summation in the second equation of the first calculation is valid as long as $\epsilon < \kappa$ (so that $|\theta| < \kappa$).
\end{proof}

\begin{lemma}
\label{lemmabetaf}
Recall from Section \ref{sectiontheoremscase12} that $\xi = \eta_0 - \cbar$\,. If $\epsilon >0$ is sufficiently small, then
\begin{equation}
\label{betaseries}
\begin{split}
\beta &= \sum_{i,j} \eta_i \theta^j \beta_{i,j} - \cbar \, h\nu
= \xi h \nu + \sum_{(i,j) \neq (0,0)} \eta_i \theta^j \beta_{i,j}\\
&= \xi h \nu + \sum_{j = 1}^\infty (\cbar + \xi) \theta^j \beta_{0,j} + \sum_{i = 1}^\infty \sum_{j = 0}^\infty \eta_i \theta^j \beta_{i,j},
\end{split}
\end{equation}
where $\beta_{i,j} \in \LL(\BB)$ is defined by
\begin{equation}
\label{betaijdef}
\begin{split}
\beta_{i,j} f(x)
&\df \frac{1}{i!} \frac{1}{j!} \left(\frac{\del}{\del b}\right)^i \left(\frac{\del}{\del\theta}\right)^j \left[e^{(\delta + \theta) v(b,x)} \; f\circ u_b(x)\right]_{b = \theta = 0}\\
&= \Coeff\left(b^i \theta^j, e^{(\delta + \theta) v(b,x)} \; f\circ u_b(x)\right).
\end{split}
\end{equation}
Here $\Coeff(X,A)$ denotes the coefficient of a multinomial $X$ in the power series expansion of $A$. Furthermore,
\begin{equation}
\label{betabound}
\|\beta_{i,j}\| \leq \rho_U^{-i} e^{(\delta + 1) \|v\|}
\end{equation}
where $\rho_U > 0$ is small enough so that $B_\C(0,\rho_U) \subset U_\C$.
\end{lemma}
Note that $\beta_{0,0} = h \nu$. This is in fact the reason that we defined $h,\nu$ as we did.
\begin{proof}
Indeed, fix $f\in\BB$. For all $x\in V_\C$ we have
\begin{align*}
(\beta + \cbar \, h \nu) f(x) &= (L_{S',\delta + \theta} - L_{S,\delta + \theta}) f(x)\\
&= \sum_{b\in S'} |b|^{q(\delta + \theta)} e^{(\delta + \theta) v(b,x)} \; f\circ u_b(x) - \sum_{b\in S} |b|^{q(\delta + \theta)} e^{(\delta + \theta) v(b,x)} \; f\circ u_b(x) \noreason \\
&= \sum_{b\in S'} |b|^{q(\delta + \theta)} \left(\sum_{i,j} b^i \theta^j \beta_{i,j} f(x)\right) - \sum_{b\in S} |b|^{q(\delta + \theta)} \left(\sum_{i,j} b^i \theta^j \beta_{i,j} f(x)\right) \noreason \\
&= \sum_{i,j} \eta_i \theta^j \beta_{i,j} f(x). \note{see below to justify interchange}
\end{align*}
The last two equalities of \eqref{betaseries} follow from the definition of $\xi$ and the fact that $\beta_{0,0} = h\nu$.

To demonstrate \eqref{betabound}, we plug in $\rho = \rho_U$ and $\rho = 1$ into \eqref{cauchy} (both with $z = 0$) for the variables $b$ and $\theta$, respectively. This yields
\begin{align*}
|\beta_{i,j} f(x)|
&\leq \rho_U^{-i} \sup_{b\in U_\C} \sup_{|\theta| \leq 1} \left| e^{(\delta + \theta) v(b,x)} \; f\circ u_b(x)\right|\\
&\leq \rho_U^{-i} e^{(\delta + 1)\|v\|} \|f\|.
\end{align*}
Note that by Lemma \ref{lemmaetabound}, this calculation also shows that the interchange of summation in the last equation of the first calculation is valid as long as $\epsilon < \min(\rho_U,1)$ (so that $\eta < \rho_U$ and $|\theta| < 1$).
\end{proof}

We now wish to prove a bound on $|\xi|$. By \eqref{alphabound}, we have $\|\alpha\| \lesssim |\theta|$, and by \eqref{betabound} and Lemma \ref{lemmaetabound}, we have $\|\beta - \xi h \nu\| \lesssim \max(\etaratio,|\theta|)$, as long as $\epsilon \leq \min(\kappa,\rho_U,1)/2$. Thus
\[
\|L' g - (g + \xi h)\| = \|(L' - L - \xi h \nu) g\| = \|(\alpha + \beta - \xi h \nu) g\| \lesssim \max(\etaratio,|\theta|).
\]
Suppose $\xi \geq 0$. Then since $h \geq \inf_{V_\R} (h/g) g$ on $V_\R$, it follows that 
\[
g + \xi h \geq \left(1 + \inf_{V_\R}(h/g)\, \xi \right)g.
\]
Now since $g$ is uniformly positive on $V_\R$, we have that 
\[
L' g \geq \lambda g \text{ on } V_\R \text{, where } \lambda = 1 + \inf_{V_\R}(h/g) \, \xi + O\big(\max(\etaratio,|\theta|)\big).
\]
Since $L'$ and $g$ are both positive, it follows that $\lambda \leq \rho(L') = 1$. Thus $\xi = O(\max(\etaratio,|\theta|))$, and the case $\xi \leq 0$ proceeds similarly. Let $C \geq 1$ be the implied constant, so that
\[
|\xi| \leq C \max(\etaratio,|\theta|) \leq C\epsilon.
\]

\begin{definition}
\label{definitionseriesfunction}
A \emph{series function} is an expression of the form 
\begin{equation}
\label{powerseries}
\ff = (f_{I,j,k}) = \sum_{I,j,k} \eta_I \theta^j \xi^k f_{I,j,k},
\end{equation}
where $f_{I,j,k} : V_\C \to \C$ are holomorphic functions independent of $\theta$ and $S'$, and $I,j,k$ are as in \eqref{mainformula}.
If $\ff$ is a series function, let 
\[
\opnorm{\ff} \df \sum_{I,j,k} (C\epsilon)^{\Sigma(I) + j + k} \|f_{I,j,k}\|.
\]
The set of series functions $\ff$ such that $\opnorm{\ff}$ is finite, which we denote as $\AA$, forms a Banach space under the norm $\opnorm{\cdot}$.
There is a natural projection map $\pi:\AA\to\BB$ defined by letting $\pi(\ff)$ be the value of the right-hand side of \eqref{powerseries}. Note that $\pi(\ff)$ depends on $S'$, while $\ff$ (being merely a formal expression) does not.
Since $\eta,|\theta|,|\xi| \leq C\epsilon$, we have 
\[
\| \pi(\ff) \| \leq \opnorm{\ff}
\] 
for every series function $\ff$. Finally, if $f\in\BB$ then we abuse notation and also let $f$ denote the series function $f = \eta_\0 \theta^0 \xi^0 f$.
\end{definition}

For $I\in \MM(\N_{\geq 1})$, let $\bfeta_I$, $\bftheta$, and $\bfxi$ denote the series functions given by the formulas 
\begin{align*} 
\bfeta_I &= \eta_I \theta^0 \xi^0 \one, & \bftheta &= \eta_\0 \theta^1 \xi^0 \one, & \bfxi &= \eta_\0 \theta^0 \xi^1 \one. 
\end{align*} 
Then by replacing $\eta_I$, $\theta$, and $\xi$ by $\bfeta_I$, $\bftheta$, and $\bfxi$ in \eqref{alphaseries} and \eqref{betaseries}, we can construct operators $\bfalpha$ and $\bfbeta$ on $\AA$ such that $\pi(\bfalpha f) = \alpha f$ and $\pi(\bfbeta f) = \beta f$ for all $f\in\BB$.


The corresponding operator norms satisfy
\begin{align*}
&\|\alpha\| \leq \opnorm{\bfalpha} \lesssim \epsilon/\kappa \asymp \epsilon &
&\text{and} &
&\|\beta\| \leq \opnorm{\bfbeta} \lesssim \epsilon/\rho_U \asymp \epsilon, 
\end{align*}
as long as $\epsilon \leq \min(\kappa,\rho_U,1)/2$. 
Recall that $\Delta = \alpha + \beta$, and let $\bfDelta \df \bfalpha + \bfbeta$. Then it follows from the above inequalities that 
\[
\|\Delta\| \leq \opnorm{\bfDelta} \lesssim \epsilon.
\] 
Thus if $\epsilon$ is sufficiently small then \eqref{Deltasufficient} holds, and thus so does \eqref{operatorequation}. Moreover, if $\opnorm{\bfDelta} < 1/\|Q\|$, then
\[
\sum_{\Qpower = 0}^\infty \opnorm[\big]{\bfDelta (Q\bfDelta)^\Qpower} \leq \opnorm{\bfDelta} \sum_{\Qpower = 0}^\infty (\|Q\|\cdot\opnorm{\bfDelta})^\Qpower < \infty,
\]
so $\sum_{\Qpower = 0}^\infty \bfDelta (Q\bfDelta)^\Qpower \in \LL(\AA)$. It follows that $\sum_{\Qpower = 0}^\infty \bfDelta (Q\bfDelta)^\Qpower g \in \AA$ and thus
\[
\sum_{\Qpower = 0}^\infty \mu \Delta (Q \Delta)^\Qpower g  = \Xi  \;\underset{\eqref{mainformula}}{\df}\; \sum_{I,j,k} c_{I,j,k} \, \eta_I \theta^j \xi^k 
\]
for some constants $c_{I,j,k}$, such that 
\[
|c_{I,j,k}| \lesssim (C\epsilon)^{-(\Sigma(I) + j + k)} \leq \epsilon^{-(\Sigma(I) + j + k)}.
\] 
This demonstrates \eqref{cijkbounds}, and we have $\Xi = 0$ by \eqref{operatorequation}.

Next we want to show that the coefficients of $1$ and $\xi$ are $c_{\0,0,0} = 0$ and $c_{\0,0,1} = 1$, respectively, and that if $\cbar = 0$ then the coefficient of $\theta$ is $c_{\0,1,0} = -\w\chi$, where $\w\chi > 0$ is as in \eqref{chidef}. Indeed, let
\[
X = \eta \; \vee \; \max(|\theta|,|\xi|)^2
\]
and recall that $A \equiv_X B$ means $B - A = O(X)$. Then by \eqref{normalization},
\begin{align*}
\Xi &\underset{X}{\equiv} \mu \Delta g = \mu \alpha g + \mu \beta g
\underset{X}{\equiv} \theta \mu \alpha_1 g + \xi \mu h \nu g + \cbar \, \theta \mu \beta_{0,1} g \\
&= \xi + (-\w\chi + \cbar \,\mu\beta_{0,1} g) \theta,
\end{align*}
which is what we wanted.


\ignore{
\section{Proof of Theorem \ref{theoremcase1}}
\label{sectioncase1}

We now prove Theorem \ref{theoremcase1} using Theorem \ref{maintheorem}. As before we fix $\epsilon,\thetabound,\zetabound > 0$ to be determined. Now let $S'$ satisfy $\zeta \df \zeta(\kappa) \leq \zetabound$ and $\epsilon \df \|S\triangle S'\| \leq \epsilon$. Then by Theorem \ref{maintheorem}, \eqref{mainformula} holds, and
\[
|c_{I,j,k}| \lesssim \thetabound^{-j} \xibound^{-k} \epsilon^{-\Sigma(I)}.
\]
Since $\cbar = 0$, we have $\xi = \eta_0$. We now wish to rewrite $\eta_i$ in terms of $\eta_{i,\mj}$ and $\theta$:

\begin{lemma}
\label{lemmaetaexpansion}
We have
\[
\eta_i = \sum_{\mj = 0}^\infty \theta^\mj \eta_{i,\mj}(S').
\]
Moreover,
\[
|\eta_{i,\mj}(S')| \lesssim \epsilon^i \kappa^{-\mj} \zeta.
\]
\end{lemma}
\begin{proof}
Indeed,
\[
\eta_i = \sum_{b\in S' - S} |b|^{q\delta} \sum_{j = 0}^\infty \frac{(q\theta\log|b|)^j}{j!} = \sum_{j = 0}^\infty \frac{q^j \theta^j}{j!} \sum_{b\in S' - S} |b|^{q\delta} \log^j|b| = \eta_i = \sum_{\mj = 0}^\infty \theta^\mj \eta_{i,\mj}(S')
\]
and
\begin{align*}
|\eta_{i,\mj}(S')| &\leq \frac{q^\mj}{\mj!} \sum_{b\in S \triangle S'} |b|^{q\delta + i} \big|\log |b| \big|^\mj\\
&= \kappa^{-\mj} \sum_{b\in S\triangle S'} |b|^{q\delta + i} \frac{(q\kappa|\log|b||)^\mj}{\mj!}\\
&\leq \epsilon^i \kappa^{-\mj} \sum_{b\in S\triangle S'} |b|^{q\delta} \exp(q\kappa|\log|b||)\\
&\asymp \epsilon^i \kappa^{-\mj} \sum_{b\in S\triangle S'} |b|^{q(\delta - \kappa)} \lesssim \epsilon^i \kappa^{-\mj} \zetabound.
\qedhere\end{align*}
\end{proof}

Now, analogously to the proof of Theorem \ref{maintheorem}, we define a \emph{series function of type 2} to be an expression of the form
\[
\sum_{I,\jprime} c_{I,\mj} \theta^\jprime \, \eta_I(S'),
\]
where $I$ ranges over multisubsets of $\N^2$, and we let
\[
\left\|\sum_{I,\jprime} c_{I,\jprime} \theta^\jprime \, \eta_I(S')\right\|_\sigma
\df \max_{I,\jprime} |c_{I,\jprime}| (2\epsilon)^{\Sigma(I)} (2\thetabound)^\jprime (2\zetabound)^{\#(I)}.
\]
By Lemma \ref{lemmaetaexpansion}, we can replace $\eta_i$ by $\sum_\mj \theta^\mj \eta_{i,\mj}$ in \eqref{mainformula} to get
\begin{equation}
\label{etaimformula}
\sum_{\substack{I,\jprime \\ \mj(I) \leq \jprime}} c_{I,\jprime} \theta^\jprime \, \eta_I(S') = 0,
\end{equation}
where the left-hand side has $\|\cdot\|_\sigma < \infty$. Moreover, we have $c_{\0,0} = c_{\0,0,0} = 0$, $c_{\0,1} = c_{\0,1,0} = -\w\chi$, and $c_{\{(0,0)\},0} = c_{\0,0,1} = 1$.

We claim next that \eqref{etaimformula} can be solved for $\theta$ yielding a formal power series
\begin{equation}
\label{thetaseries}
\theta = \sum_{\substack{I \\ \mj(I) \leq \#(I) - 1}} c_I \, \eta_I(S'),
\end{equation}
which we will then show converges. Indeed, writing \eqref{etaimformula} as
\begin{equation}
\label{solvefortheta}
\theta = \T_{S'}(\theta) \df -\frac{1}{\w\chi} \left[\sum_{\jprime = 2}^\infty c_{\0,\jprime} \theta^\jprime + \sum_{\jprime = 0}^\infty \sum_{\substack{I\neq \0 \\ \mj(I) \leq \jprime}} c_{I,\jprime} \theta^\jprime \, \eta_I(S')\right]
\end{equation}
yields a recursive formula for $\theta$. Taking the $n$th iterate yields
\begin{equation}
\label{iteraten}
\theta = \T_{S'}^n(0) + O(\theta^{n + 1})
\end{equation}
at least on the formal level. We claim that all nonzero terms $c_I \, \eta_I(S')$ on the right-hand side of \eqref{iteraten} satisfy $\mj(I) \leq \#(I) - 1$. Indeed, suppose that this is true for $n$, and we will prove it for $n+1$. Since $\sum_{\Qpower = 0}^\infty \Delta (Q\Delta)^\Qpower g$ is a series function, each term comes from an expression $c_{I,\jprime} \theta^\jprime \, \eta_I(S')$, where $\mj(I) \leq \jprime$, and either $\jprime\geq 2$ or $I\neq \0$. The induction hypothesis yields that if $c_{I'} \eta_{I'}(S')$ is a nonzero term in $\theta^\jprime$ (corresponding to a term $c_{I,\jprime} c_{I'} \eta_{I+I'}(S')$ on the right-hand side of \eqref{iteraten}), then $\mj(I') \leq \#(I') - \jprime$ and thus $\mj(I + I') \leq \#(I')$. If $\#(I) \geq 1$, then $\mj(I + I') \leq \#(I + I') - 1$ and we are done. Otherwise, $I = \0$ and thus $\mj(I) = 0$, so $\mj(I + I') = \mj(I') \leq \#(I') - \jprime < \#(I + I') - 1$. Thus, taking the limit of \eqref{iteraten} as $n\to\infty$ (with each coefficient eventually constant) yields \eqref{thetaseries}.

To prove that the series \eqref{thetaseries} in fact converges, fix $\gamma > 0$ and take $a,b\in\AA$ with $\|a\|_\sigma,\|b\|_\sigma \leq \thetabound$ and $\|b - a\|_\sigma \leq \gamma$. Then
\[
\|\T_{S'}(b) - \T_{S'}(a)\|_\sigma \leq \sum_{I,\jprime} |c_{I,\jprime}| \jprime\thetabound^{\jprime - 1} \gamma \epsilon^{\Sigma(I)} \xibound^{\#(I=0)} \lesssim \sum_\jprime \frac{\jprime\thetabound^{\jprime-1}}{\kappa^\jprime} \gamma ([\jprime\geq 2] + \zetabound) \lesssim (\thetabound + \zetabound) \gamma
\]
and thus if $\thetabound,\zetabound$ are small enough, we have $\|\T_{S'}(b) - \T_{S'}(a)\|_\sigma \leq (1/2) \|b - a\|_\sigma$. Moreover, $\|\T(0)\|_\sigma = \sum_{I\neq \0} |c_{I,0}| \cdot \epsilon^{\Sigma(I)} \zetabound^{\#(I)} \leq C \zetabound$ for some $C > 0$ and thus by induction, for all $n$ we have $\|\T^n(0)\|_\sigma \leq 2 C \zetabound(1 - 2^{-n}) < \thetabound$ and $\|\T^{n + 1}(0) - \T^n(\0)\|_\sigma \leq 2^{-n} C \epsilon$, as long as $\zetabound \leq \thetabound/2C$. So the limit of $\T^n(0)$ exists and has $\|\cdot\|_\sigma$-norm $\leq 2\zetabound$.

}


%

\section{Simplifications in special cases}
\label{sectionvanishing}

In some special cases, we can make some simplifications to \eqref{mainformula}, which will be used in Section \ref{sectiongauss} below.

\begin{proposition}
\label{propositionvanishing}
Let $S,\delta,S'$ be as in Theorem \ref{maintheorem}. Then \eqref{mainformula} can be simplified in the following ways:
\begin{itemize}
\item[(i)] If $S = \emptyset$, then (possibly after renormalizing $g$ and $\mu$) we have $\alpha = 0$, $L_1 = 0$, $Q_1 = I$, $g = h$, $\mu = \nu$, $c = \cbar = 1/\nu h$, $L = c g \mu$, $R = 0$, $Q = I$, and
\begin{equation}
\label{Semptyset}
\Xi = \sum_{\Qpower \geq 1} \nu \beta^\Qpower h.
\end{equation}
\item[(ii)] If $v(0,\cdot) = 0$, then $\beta_{i,j} = 0$ for all $i,j$ with $j > i$. If furthermore $v(\cdot,0) = 0$, then $\nu \beta_{i,j} = 0$ for all $i,j$ with $j > 0$. 
If both hypotheses hold, and in addition $S = \emptyset$, then $c_{I,j,k} = 0$ whenever $j \geq \Sigma(I)$ and $(I,j) \neq (\0,0)$. In particular, $c_{I,j,k} = 0$ for all $I,j,k$ with $j \geq \Sigma(I)$ and $k = 0$, and for all $I,j,k$ with $j > \Sigma(i)$. 
\item[(iii)] If $v(0,\cdot) = 0$ and $v(\cdot,0) = 0$, then $\nu \beta_{i,j} h = 0$ for all $(i,j) \neq (0,0)$. If furthermore $S = \emptyset$, then $c_{I,j,0} = 0$ for all $I,j$ such that $\#(I) < 2$. 
\item[(iv)] If $\delta = 0$, then $\beta_{i,0} g = \beta_{i,0} h = 0$ for all $i > 0$. This implies that $c_{I,0,k} = 0$ for all $I,k$ such that $I\neq \0$.
\end{itemize}
\end{proposition}
Since (i) is a straightforward consequence of the definitions, we proceed to the proofs of (ii)-(iv).
\begin{proof}[Proof of (ii)]
If $v(0,\cdot) = 0$, then we can write $v(b,x) = b w(b,x)$ for all $b,x$, for some holomorphic function $w$. Thus for all $i,j,f,x$ with $j > i$,
\[
\beta_{i,j} f(x) = \frac1{j!}\Coeff\left(b^i,e^{\delta v(b,x)} b^j w^j(b,x) f \circ u_b(x)\right) = 0.
\]
If furthermore $v(\cdot,0) = 0$, then for $i,j,f$ with $j > 0$,
\[
\nu \beta_{i,j} f = \Coeff\big(b^i\theta^j,e^{(\delta + \theta) v(b,0)} f(u_b(0))\big) = \Coeff\big(b^i\theta^j, f(u_b(0))\big) = 0.
\]
If $v(0,\cdot) = 0$, $v(\cdot,0) = 0$, and $S = \emptyset$, then applying (i) gives us 
\[
\Delta = \beta = \xi h \nu + \sum_{i = 1}^\infty \sum_{j = 0}^i \eta_i \theta^j \beta_{i,j},
\]
and
\[
\mu \Delta = \nu \beta = \xi \nu h \nu + \sum_{i = 1}^\infty \eta_i \nu\beta_{i,0} .
\]
One proves by induction that for every $p \in \N$
\[
\nu \beta^p = \sum_{k=0}^\infty c_k \xi^k \nu + \sum_{I\in\MM(\N_{\geq 1})} \sum_{j=0}^{\#(I)-1} \sum_{k=0}^\infty \eta_I \theta^j \xi^k \sigma_{i,j,k}
\]
Then multiplying both sides of the above equation on the right by $h$ 
and using \eqref{Semptyset} leads to 
\[
c_{I,j,k} = 0
\]
whenever $j \geq \Sigma(I)$ and $(I,j) \neq (\0,0)$.
\end{proof}
\begin{proof}[Proof of (iii)]
If $v(0,\cdot) = 0$, then $h = \one$. If furthermore $v(\cdot,0) = 0$, then for all $i,j$
\begin{equation*}
\nu \beta_{i,j} h = \Coeff\left(b^i \theta^j,e^{(\delta + \theta) v(b,0)} \one(u_b(0))\right) = \Coeff(b^i \theta^j,1) = \big[i = j = 0\big].
\end{equation*}
Here we use the Iverson bracket notation: $[\Phi] = 1$ when $\Phi$ is true and $[\Phi] = 0$ when $\Phi$ is false. By \eqref{betaseries}, it follows that $\nu \beta h = \xi$.

If furthermore $S = \emptyset$, then combining with part (i) gives
\[
\Xi = \xi + \sum_{\Qpower = 2}^\infty \nu \beta^\Qpower h,
\]
while by part (ii), all terms $\theta^i\cbar \beta_{i,0}$ appearing in $\beta$ vanish, and thus $\beta$ is of the sum of terms with factors $\xi$ and $\eta_i$. It follows that every term $c_{I,j,k} \, \eta_I \theta^j \xi^k$ appearing in the above series satisfies $\#(I) + k \geq 2$. In particular, if $\#(I) < 2$ then $c_{I,j,0} = 0$.
\end{proof}
\begin{proof}[Proof of (iv)]
If $\delta = 0$ then $L_1 \one = \#(S) \one$ and $h = \one$, and thus $L \one = (\#(S) + \cbar\,)\one$. It follows that $\#(S) + \cbar = 1$ and (after renormalizing) $g = \one$. Now for all $i$,
\[
\beta_{i,0} \one(x) = \Coeff\left(b^i,e^{(0 + 0) v(b,x)} \one \circ u_b(x)\right) = \Coeff(b^i,1) = \big[i = 0\big].
\]
Now, the coefficient $c_{I,0,k}$ is the sum of all products of the form $\mu \beta_{i_1,0} Q \cdots Q \beta_{i_\Qpower,0} g$ such that $k = \#(\ell : i_\ell = 0)$ and $I(i) = \#(\ell : i_\ell = i)$ for all $i$. For each such product, either $i_1 = \ldots = i_\Qpower = 0$, in which case $I = \0$, or there exists $\ell = 1,\ldots,\Qpower$ such that $i_\ell > 0$, and either $\beta_{i_\ell,0} Q \beta_{0,0}$ or $\beta_{i_\ell,0} g$ is a factor of the term in question. But since $Q \beta_{0,0} = Q h \nu = \one \nu$ and $g = \one$, both of these factors vanish, and thus $c_{I,0,k} = 0$.
\end{proof}

The next two propositions are proving slightly different things. Proposition \ref{propositionrationalv2} has the advantage that it applies to perturbations of a similarity IFS with more than one element. Proposition \ref{propositionrational} has the advantage that it applies to a larger class of non-Gauss PACIFSes, such as those defined over the system $(U,V,u,v,q)$ where $U = V = (-1/3,1/3)$, $u(b,x) = \frac{b}{1 + x}$, $v(b,x) = -2\log(1 + x)$, and $q = 1$.

In what follows, recall that $\Q[x]$ denotes the ring of polynomials in the variable $x$ with coefficients in $\Q$.

\begin{proposition}
\label{propositionrationalv2}
Let $S,\delta,S'$ be as in Theorem \ref{maintheorem}, and suppose that $S$ is finite, $(u_b)_{b\in S}$ consists entirely of similarities, $v(0,\cdot) = 0$, and $\Coeff(b^i, u_b(x)) \in \Q[x]$ for all $i$. Let
\[
R \df \Q(\lambda_b,\lambda_b^\delta)[\log(\lambda_b),c_b,\delta]
\]
where $\lambda_b$ is the contraction ratio of $u_b$, $c_b = u_b(0)$, and the extensions are taken over all $b\in S$. Then $c_{I,j,k} \in R$ for all $I,j,k$. In particular, if $S = \emptyset$, then $c_{I,j,k} \in \Q[\delta]$ for all $I,j,k$.
\end{proposition}
\begin{proof}
We claim that $\alpha_j,\beta_{i,j},Q$ preserve $R[x]$, and that $\mu$ sends $R[x]$ to $R$. Indeed, if $f_k(x) = x^k$, then
\[
\alpha_j f_k(x) = \frac1{j!} \sum_{b\in S} \lambda_b^\delta (\pm\lambda_b x + c_b)^k \log^j(\lambda_b) \in R[x]
\]
and in particular $L f_k(x) = \alpha_0 f_k(x)$ is a polynomial of degree $k$ with leading coefficient $a_k = \sum_{b\in S} \pm\lambda_b^{k + \delta} \in \Q[\lambda_b,\lambda_b^\delta]$, which by the Moran--Hutchinson equation (e.g. \cite{Hutchinson}) satisfies $|a_k| < 1$ when $k > 0$. Since $\mu L = \mu$, we have
\[
\mu f_k = \mu L f_k = a_k \mu f_k + \sum_{k' < k} a_{k'}^{(k)} \mu f_{k'}
\]
which, together with the equality $\mu \one = 1$ (cf. \eqref{normalization}), yields a recursive formula for $\mu f_k$ proving that $\mu f_k \in R$, and thus $\mu f \in R$ for all $f\in R[x]$. Similarly, $Q = Q L + I - c g \mu$ and thus since $g = \one = f_0$ and $c = 1$,\Footnote{It is easy to see that $L \one = \one$, so the normalization \eqref{normalization} guarantees $g = \one$.}
\[
Q f_k = a_k Q f_k + \sum_{k' < k} a_{k'}^{(k)} Q f_{k'} + f_k - f_0 \mu f_k
\]
whicy yields a recursive formula for $Q f_k$ proving that $Q f_k(x) \in R[x]$, and thus that $Q$ preserves $R[x]$. Now since $v(0,\cdot) = 0$,
\[
\beta_{i,j} f_k(x) = \sum_{i' = 0}^i \frac1{i'!} \Coeff\big(\theta^j,(\delta + \theta)^{i'}\big) \Coeff\big(b^i,b^{i'} w^{i'}(b,x) u_b^k(x)\big)
\]
where $v(b,x) = b w(b,x)$ as above. Since $\Coeff(b^i, u_b(x)) \in \Q[x]$, we have $\Coeff(b^i,u_b'(x)) \in \Q[x]$ and thus $\Coeff(b^i,e^{v(b,x)}) \in \Q[x]$ for all $i$. Since $e^{v(0,x)} = 1$, the Taylor expansion of $\log(x)$ around $x = 1$ shows that  $\Coeff(b^i,v(b,x)) \in \Q[x]$ for all $i$. Thus $\beta_{i,j}$ preserves $R[x]$, which completes the proof.
\end{proof}

\begin{definition}
Let $R$ be a subring of $\R$, e.g. $\Q$. An analytic function $f:U\to \R$, where $U$ is a neighborhood of $\0$ in $\R^d$, is \emph{$R$-analytic} if the coefficients of the Taylor expansion of $f$ at $\0$ are all in $R$.
\end{definition}

Note that $R$-analyticity is highly sensitive to the point that the Taylor series is expanded around; if $f$ is $R$-analytic then $x\mapsto f(x - a)$ may not be $R$-analytic even if $a\in R^d$.

\begin{proposition}
\label{propositionrational}
Let $S,\delta,S'$ be as in Theorem \ref{maintheorem}, and suppose that $S = \emptyset$, that $v(0,0) = 0$, and that $\Coeff(b^i x^k , u_b(x)) \in \Q$ for all $i,k$. Then $c_{I,j,k} \in \Q[\delta]$ for all $I,j,k$.
\end{proposition}

\begin{proof}
By \vanishinga, $\Xi = \sum_{\Qpower\geq 1} \nu \beta^\Qpower h$. Thus, all coefficients $c_{I,j,k}$ in \eqref{mainformula} can be written as linear combinations of expressions of the form $\nu \beta_{i_1,j_1}\cdots \beta_{i_\Qpower,j_\Qpower} h$, where $\Qpower \geq 1$. In turn, we can write
\[
\beta_{i,j} = \sum_{i' = 0}^i \sum_{k = 0}^{i'} M_{\psi_{i',k}} f_{i - i',j} \sigma_k, 
\]
where
\begin{align*}
f_{i',j}(x) &= \Coeff\left( b^{i'} \theta^j , e^{(\delta + \theta) v(b,x)}\right),\\
\sigma_k f &= \Coeff\big(x^k, f(x)\big),\\
\psi_{i',k}(x) &= \Coeff\big(b^{i'}, u_b^k(x)\big)
\end{align*}
Thus, all coefficients can be written as linear combinations of products of expressions of the form $\sigma_i M_{\psi_{j,k}} f_{\ell,m}$. Here we use the fact that $\sigma_0 = \nu$ and $f_{0,0} = h$.

To show that $\sigma_i M_{\psi_{j,k}} f_{\ell,m} \in \Q[\delta]$, first note that since $u$ is $\Q$-analytic and $v(0,0) = 0$, it follows that $(b,x) \mapsto e^{v(b,x)} = \pm b^{-q} u_b'(x)$ is $\Q$-analytic and sends $(0,0)$ to $1$; thus $v$ is $\Q$-analytic. Again using the fact that $v(0,0) = 0$, it follows that $(b,\theta,x) \mapsto e^{(\delta + \theta) v(b,x)}$ is $\Q[\delta]$-analytic. This shows that $f_{\ell,m}$ and $\psi_{j,k}$, and thus $M_{\psi_{j,k}} f_{\ell,m}$, are $\Q[\delta]$-analytic, so $\sigma_i M_{\psi_{j,k}} f_{\ell,m}$, i.e. the $i$th coefficient of $M_{\psi_{j,k}} f_{\ell,m}$, is in $\Q[\delta]$.
\end{proof}

The next proposition is not strictly necessary for our purposes, but shows how our formula is a generalization of the Moran--Hutchinson equation.

\begin{proposition}
\label{propositionhutchinson}
If $(u_b)_{b\in U}$ consists entirely of similarities, then \eqref{mainformula} reduces to the Moran--Hutchinson equation $e^{P'} = 1$, where $P' = P(S',s)$. Specifically,
\[
\Xi = \frac{e^{P'} - 1}{2 - e^{P'}} \cdot
\]
\end{proposition}
\begin{proof}
All the operators $A = L,Q,R,\alpha_j,\beta_{i,j}$ satisfy $A\one = [A]\one$ for some $[A]\in\R$, and the map $A\mapsto [A]$ is a ring homomorphism. Similarly, if we write $[\sigma] = \sigma \one$ and $[r \one] = r$, then we get
\[
\Xi = \sum_{\Qpower = 0}^\infty [\nu] [\Delta] ([Q] [\Delta])^\Qpower [g] = \frac{[\Delta]}{1 - [Q] [\Delta]},
\]
since \eqref{normalization} implies $[g] = [\nu] = 1$. Moreover, since $L g = g$, we have $[L] = 1$ and thus $[R] = 0$, $[Q] = 1$. So $[\Delta] = [L'] - 1 = e^{P'} - 1$, which completes the proof.
\end{proof}

\section{Solving \eqref{mainformula}: a motivating computation}
\label{sectionlogloglog}


Although the proof of Theorem \ref{maintheorem} shows that \eqref{operatorequation} can be converted into the power series equation \eqref{mainformula}, there remains the question of how to solve this equation for $\theta$, particularly since $\eta_i$ and $\xi$ both depend on $\theta$. In some cases this is relatively easy, but in some cases more tools are needed. In this section we develop a tool that will help us solve for $\theta$ in the more difficult cases.

We start by considering a special case consisting of similarities, so that we can use the simpler Moran--Hutchinson equation in place of \eqref{mainformula}. Namely, for each $\lambda,B > 0$ wth $\lambda + B \leq 1$, we can consider a similarity IFS on $\R$ consisting of two elements with contraction ratios $\lambda$ and $B$, with distinct fixed points. For $B$ small, this system is a perturbation of the system consisting of ony one similarity contraction of contraction ratio $\lambda < 1$. On the other hand, the dimension $\theta$ of the limit set of this IFS is given by the Moran--Hutchinson equation:
\[
\lambda^\theta + B^\theta = 1.
\]
We want to analyze the behavior of $\theta$ when $\lambda$ is fixed and $B \to 0$. To this end, we note that
\begin{equation}
\label{Btheta}
B^\theta = 1 - \lambda^\theta = \theta f(\theta)
\end{equation}
for an analytic function $f$ depending on $\lambda$ such that $f(0) = \log(1/\lambda) > 0$. Taking logarithms yields
\[
-C\theta = \log(\theta) + \log f(\theta),
\]
where here and in the rest of this section we use the notation
\begin{equation}
\label{CDE}
\begin{split}
C &\df -\log(B) = \log(1/B),\\
D &\df \log(C) = \log\log(1/B),\\
E &\df \log(D) = \log\log\log(1/B).
\end{split}
\end{equation}
We now follow the heuristic of changing variables in such a way so that the new variable is ``closer to being bounded from above and below'' than the previous variable. Thus, let $\gamma \df C\theta > 0$. Then
\[
\gamma = -\log(\gamma/C) - \log f(\theta) = D - \log(\gamma) - \log f(\theta).
\]
Let $\beta \df D - \gamma$. Then
\[
\beta = \log(D - \beta) + \log f(\theta) = E + \log(1 - \beta/D) + \log f(\theta).
\]
Let $\alpha \df \beta - E$, and note that
\begin{equation}
\label{thetaalpha}
\theta = \frac{D - E - \alpha}{C} \cdot
\end{equation}
Then
\begin{equation}
\label{alpharecursive}
\alpha = \log\left(1 - \frac{E}{D} - \frac{\alpha}{D}\right) - \log f\left(\frac{D - E - \alpha}{C}\right)
\end{equation}
and thus 
\begin{equation}
\label{alphaF}
\alpha = F\left(\frac ED,\frac 1D,\frac DC,\alpha\right),
\end{equation} 
where
\[
F(x,y,z,w) \df \log(1 - x - y w) - \log f\big(z(1 - x - y w)\big).
\]
This concludes our change of variables, and we have rewritten \eqref{Btheta} as \eqref{alphaF}.

Note that $F(\0,w) = \alpha_0 \df -\log f(0)$ and $F_{|4}(\0,w) = 0$ for all $w\in \R$, and $F$ is analytic on a neighborhood of $\{0\}^3 \times \R$. Thus, $\alpha_0 - F(\0,\alpha_0) = 0$ and $\frac{\del}{\del w} [w - F(0,0,0,w)] = 1$. So by the implicit function theorem, the equation $w - F(w,x,y,z) = 0$ can be solved analytically for $w$ in terms of $(x,y,z)$ in a neighborhood of $(0,0,0)$. Since $\frac ED,\frac 1D,\frac DC \to 0$ as $B \to 0$, it follows that $\alpha$ can be written as a power series in $\frac ED,\frac 1D,\frac DC$ whenever $B$ is sufficiently small. By \eqref{thetaalpha}, we have
\[
\theta = \frac{1}{C}\left[D - E + \sum_{j = 0}^\infty \sum_{k = -j}^\infty \sum_{\ell = 0}^{j + k} c_{j,k,\ell} \frac{E^\ell}{C^j D^k}\right].
\]
The coefficients $c_{j,k,\ell}$ can be computed recursively using the formula \eqref{alpharecursive}.

The following lemma generalizes the above calculation:
\begin{lemma}
\label{lemmathetageneral}
If $f:(\aa,\theta,\xi)\mapsto f(\aa,\theta,\xi)$ is analytic in a neighborhood of $A\times \{(0,0)\} $, with $f(\aa,0,0) > 0$ for all $\aa\in A \subset \R^d$, then for all $B$ sufficiently small and for all $\aa \in A$, the equation
\begin{equation}
\label{thetageneral}
B^\theta = \theta f(\aa,\theta,B^\theta)
\end{equation}
has a unique solution:
\begin{equation}
\label{thetageneralsolved}
\theta = \frac1C \left[D - E - \sum_{j = 0}^\infty \sum_{k = -j}^\infty \sum_{\ell = 0}^{j + k} f_{j,k,\ell}(\aa) \frac{E^\ell}{C^j D^k} \right]
\end{equation}
(cf. \eqref{CDE}) for some functions $f_{j,k,\ell}$ analytic on a complex neighborhood $A_\C$ of $A$, such that
\[
\|f_{j,k,\ell}\| \lesssim \epsilon^{-(j + k + \ell)}
\]
for some $\epsilon > 0$. If $f$ is constant then
\[
\theta = \frac1C \left[D - E - \sum_{k = 0}^\infty \sum_{\ell = 0}^k c_{k,\ell} \frac{E^\ell}{D^k} \right]
\]
and if $f(\aa,0,0) = 1$ then $f_{0,k,0}(\aa) = 0$ for all $k$. If $f(\aa,0,0) = 1$ and $f$ is $\Q$-analytic, then so are $f_{j,k,\ell}$.
\end{lemma}
\begin{proof}
First note that if $\alpha$ is defined as the unique solution to \eqref{thetaalpha}, i.e. $\alpha \df D - E - C\theta$, then
\[
B^\theta = e^{-(D - E - \alpha)} = \frac DC e^\alpha.
\]
Thus, repeating the above calculations shows that the equation \eqref{thetageneral} is equivalent to 
\[
\alpha = F\left(\aa,\frac ED,\frac 1D,\frac DC,\alpha\right),
\] 
where
\[
F(\aa,x,y,z,w) \df \log(1 - x - y w) - \log f\big(\aa, z(1 - x - y w), z e^w\big).
\]
As before, we have $F(\aa,\0,w) = \alpha_0(\aa) \df -\log f(\aa,0,0,0)$ and $F_{|5}(\aa,\0,w) = 0$ for all $w\in\R$, and $F$ is analytic in a neighborhood of $A \times \{\0\} \times \R$. So as before, by the implicit function theorem the equation $w = F(\aa,x,y,z,w)$ can be solved analytically for $w$ in terms of $(\aa,x,y,z)$ in a neighborhood of $A\times \{\0\}$, and thus $\alpha$ can be written as a power series in $\frac ED,\frac 1D,\frac DC$ with coefficients in $\H(U_\C)$ whenever $B$ is sufficiently small, for some complex neighborhood $A_\C$ of $A$. Applying \eqref{thetaalpha} demonstrates \eqref{thetageneralsolved}.

If $f$ is constant, then $F(\aa,x,y,z,w)$ is constant with respect to $\aa,z$, so $\alpha$ can be written as a power series in $\frac ED,\frac 1D$. If $f(\aa,0,0) = 1$, then $F(\aa,0,y,0,0) = 0$, which implies that the solution of $w = F(\aa,x,y,z,w)$ vanishes on the line $x = z = 0$. It follows that $f_{0,k,0}(\aa) = 0$ for all $k$.
\end{proof}

\section{Gauss IFS Examples}
\label{sectiongauss}

We now use Theorem \ref{maintheorem} to compute and estimate the Hausdorff dimensions of sets of the form
\[
F_E = \{[0;\an_1,\an_2,\ldots] : \an_1,\an_2,\ldots \in E\}
\]
where $E \subset \N$, and $[0;n_1,n_2,\ldots]$ represents the continued fraction expansion with partial quotients $n_1,n_2,\ldots$. To this end, we let $(U,V,u,v,q)$ be as in Subsection \ref{subsectiongauss} and for each $E \subset \N$ let $S(E) = \{1/\an : \an \in E\}$ be the Gauss PACIFS as Subsection \ref{subsectiongauss}, so that  $\Lambda_{S(E)} = F_E$.

\begin{remark}
\label{remarkgaussvalues}
Note that $h(\cdot) = 1$, that $v(0,\cdot) = 0$ and $v(\cdot,0) = 0$, and that $u$ is $\Q$-analytic.

If $E = \N$, then $\delta = 1$, and since $\mu h = \nu g = 1$, we have $g(x) = 1/(1 + x)$, and $\mu$ is the Lebesgue measure on $[0,1]$. In this case, we have $1/c = \mu g = \log(2)$.\Footnote{Note that $g$ is usually normalized so that $\mu g = 1$, i.e. $g(x) = \frac{1}{\log(2)(1 + x)}$; however, we find the normalization \eqref{normalization} more convenient.}
\end{remark}

In what follows we will consider various sequences of sets $E_N \to E$ (we recall that this notation means that the characteristic functions converge pointwise). In each case we let $S = S(E)$ and $S' = S_N = S(E_N)$. Similarly, we write
\begin{align*}
\delta_E &\df \delta_S, &
\delta_N &\df \delta_{S_N}, &
\delta &\df \lim_{N\to\infty} \delta_N,\\
\theta &= \delta_N - \delta, &
P(E,s) & \df P(S(E),s), &
F_N &\df F_{E_N} = \Lambda_{S_N}.
\end{align*}
Note that
\begin{align*}
\eta_i &= \sum_{\an\in E_N - E} \frac{1}{\an^{2(\delta + \theta) + i}} \cdot
\end{align*}
Our goal is now to express $\eta_i$ in terms of $N$ and $\theta$, as well as information about the sequence $(E_N)$. To do this we consider various cases for the sequence $(E_N)$.

\subsection{Polynomial sequences}

The two sequences considered in the introduction,
\begin{align*}
E_N &= \{1,\ldots,N\} \to E = \N, &
E_N &= \{N,N+1,\ldots\} \to E = \emptyset,
\end{align*}
share the property that the symmetric difference $E\triangle E_N$ is a tail of $\N$. Evidently, this means we need similar methods to compute or estimate $\eta_i$ in these two cases; specifically, we need the Euler--Maclaurin formula. It turns out that the Euler--Maclaurin formula is also useful in the more general case where $E\triangle E_N$ is the tail of a polynomial sequence in $\N$ (and $(E_N)$ is either ascending or descending).


\begin{theorem}[Euler--Maclaurin formula, {\cite{Apostol,Lehmer}}]
\label{theoremEM}
Given natural numbers $M<N$ and $p$, let $f$ be a $p$-times continuously differentiable function defined on $[M,N]$. Then
\[
\sum_{\an = M}^N f(\an) = \int_M^N f(x) \; \dee x + \frac{f(M) + f(N)}{2} + \sum_{i = 1}^{\lfloor p/2 \rfloor} \frac{B_{2i}}{(2i)!} \big( f^{(2i - 1)}(N) - f^{(2i - 1)}(M) \big) + R_p,
\]
where $(B_i)$ is the sequence of Bernoulli numbers, and
\[
|R_p| \leq \frac{2\zeta(p)}{(2\pi)^p} \int_M^N |f^{(p)}(x)| \; \dee x.
\]
\end{theorem}

For $\alpha > 0$, letting $f(x) = x^{-(1 + \alpha)}$ and taking the limit as $N\to\infty$ yields the following corollary:

\begin{corollary*}
Fix $p \in \N$, $C < \infty$, and $\alpha > 0$ such that $|\alpha| \leq C$. Then for all $N\in\N$,
\begin{align*}
\sum_{\an \gepm N} \frac{1}{\an^{1+\alpha}} &= \frac{N^{-\alpha}}{\alpha} + \sum_{i=1}^{p - 1} P_i^\pm(\alpha) N^{-(i + \alpha)} + O_{p,C}( N^{-(p + \alpha)})
\end{align*}
where $\gepm$ can be taken to mean either $\geq$ or $>$, with the $\pm$ on the right-hand side depending on which choice is made, and $(P_i^\pm)_{i\geq 1}$ is an explicit sequence of polynomials (with rational coefficients):
\[
P_i^\pm(\alpha) = \frac{B_i^\pm}{i!} \binom{-(1 + \alpha)}{i - 1}
\]
where we use the convention $B_1^+ = 1/2$, $B_1^- = -1/2$ for the Bernoulli sequence. Note that $P_1^\pm(\alpha) = \pm 1/2$.
\end{corollary*}

By taking a formal limit as $p\to \infty$ we can think of this as an \emph{asymptotic expansion} of the sum $\sum_{\an \gepm N} \frac{1}{\an^{1+\alpha}}$:
\begin{equation}
\label{EM}
\sum_{\an \gepm N} \frac{1}{\an^{1+\alpha}} \equiv \frac{N^{-\alpha}}{\alpha} + \sum_{i = 1}^{\to\infty} P_i^\pm(\alpha) N^{-(i + \alpha)}
\end{equation}
where $A \equiv \sum_{i = i_0}^{\to\infty} a_i x_i$ means that for all $p \geq i_0$,
\[
A = \sum_{i = i_0}^{p - 1} a_i x_i + O_p(x_p).\Footnote{In the sequel, multiple summations are handled as follows: $A \equiv \sum_{i = i_0}^{\to\infty} \sum_{j = j_0}^{\to\infty} a_{ij} x_i y_j$ means that for all $p \geq i_0$, $q\geq j_0$, \[A = \sum_{i = i_0}^{p - 1} \sum_{j = j_0}^{q - 1} a_{ij} x_i y_j + O_p(x_p) + O_q(y_q).\]}
\]
However, note that the series \eqref{EM} does not actually converge (due to the explosion of the sequence of Bernoulli coefficients $(B_i)_{i = 1}^\infty$).

Now let $(s_n)$ be a sequence defined by a polynomial of degree $d$ and leading coefficient $a > 0$, say
\[
s_n = a (n^d + b_1 n^{d - 1} + \ldots + b_d).
\]
Fix $\alpha > 0$. Then for all $n$ sufficiently large,
\begin{align*}
\frac{1}{s_n^{(1 + \alpha)/d}}
&= \frac{(1 + b_1 n^{-1} + \ldots + b_d n^{-d})^{-(1+\alpha)/d}}{a^{(1+\alpha)/d} n^{1 + \alpha}}
= \sum_{i = 0}^\infty \frac{f_i(\alpha)}{n^{1 + i + \alpha}}
\end{align*}
for some entire functions $f_i$, with $f_0(\alpha) = a^{-(1 + \alpha)/d}$. Applying \eqref{EM} shows that
\begin{equation}
\label{EM2}
\sum_{n \gepm N} \frac{1}{s_n^{(1+\alpha)/d}} \equiv a^{-(1 + \alpha)/d} \frac{N^{-\alpha}}{\alpha} + \sum_{i = 0}^{\to\infty} \what f_i(\alpha) N^{-(i + \alpha)}
\end{equation}
for some functions $\what f_i$ holomorphic on $\C_{> -1}$, where for each $r\in\R$,
\[
\C_{> r} \df \{z\in \C : \Re z > r\}.
\]
We are now ready to prove Theorems \ref{theoremFleqN} and \ref{theoremFgeqN} from the introduction:

\begin{proposition}
\label{propositionpolysequence}
Let $(s_n)$ be a polynomial sequence of degree $d$, let $F \subset \N$ be an empty or strongly regular set disjoint from $\{s_M,s_{M+1},\ldots\}$ for some $M$, and consider the following cases:
\begin{align} \label{goingup} \tag{$\nearrow$}
E_N &= F \cup \{s_M,\ldots,s_N\} \hspace{0.1 in} \to E = F \cup \{s_M,s_{M+1},\ldots\}\\ \label{goingdown} \tag{$\searrow$}
E_N &= F \cup \{s_N,s_{N+1},\ldots\} \to E = F
\end{align}
Note that $\delta = \lim\delta_N = \max(\delta_E,1/2d)$, and for convenience of notation write $\deltabar = 2d\delta - 1 \geq 0$ and $\thetabar = 2d\theta$.
\begin{itemize}
\item[(i)] Suppose $\delta > 1/2d$, i.e. $\deltabar > 0$. Then
\begin{equation}
\label{HDFleqNgeneral}
\HD(F_N) = \delta_N \equiv \delta + \sum_{i = 0}^{\to\infty} \sum_{k = 1}^{\to\infty} \sum_{j = 0}^{k - 1} c_{i,k,j} \frac{\log^j(N)}{N^{i + k\deltabar}},
\end{equation}
with $c_{0,1,0} = \pm 1/\w\chi$. (Here and hereafter $\pm$ represents $+$ in the case \eqref{goingdown}, and $-$ in the case \eqref{goingup}.) In the case \eqref{goingup}, if $F = \emptyset$, $s_n = n$, and $M = 1$, then \eqref{HDFleqNgeneral} reduces to \eqref{HDFleqN}.
\item[(ii)] Suppose $\delta = 1/2d$, i.e. $\deltabar = 0$. Then
\begin{equation}
\label{HDFgeqNgeneral}
\begin{split}
\HD(F_N) = \delta_N \equiv \frac{1}{2d} & + \frac{1}{A d\log(N)}\left[\log\log(N) - \log\log\log(N)\phantom{\sum_{j=0}^\infty}\right.\\
&\left.+ \sum_{i = 0}^{\to\infty} \sum_{j = [i > 0]}^\infty \sum_{k = -j}^\infty \sum_{\ell = 0}^{j + k} c_{i,j,k,\ell} \frac{\log^\ell\log\log(N)}{N^i \log^j(N) \log^k\log(N)} \right]
\end{split}
\end{equation}
for some constants $c_{i,j,k,\ell}$, where $A = 2$ if $\cbar > 0$ and $A = 1$ if $\cbar = 0$. If $F = \emptyset$ and $a = 1$, then $c_{i,j,k,\ell} \in \Q$ for all $i,j,k,\ell$. In the case \eqref{goingdown}, if $F = \emptyset$ and $s_n = n$, then \eqref{HDFgeqNgeneral} reduces to \eqref{HDFgeqN}, and the constants $c_{k,\ell}$ and $c_{i,j,k,\ell}$ have the values specified in Theorem \ref{theoremFgeqN}.
\end{itemize}
\end{proposition}
Parts (i) and (ii) imply Theorems \ref{theoremFleqN} and \ref{theoremFgeqN}, respectively.
\begin{proof}
First we observe by direct calculation that for all $N$, the sets $\{s_M,\ldots,s_N\}$ and $\{s_N,s_{N + 1},\ldots\}$ are strongly regular. Since the union of two strongly regular sets is strongly regular, it follows that $E$ is empty or strongly regular, and the $E_N$s are strongly regular. Thus Theorem \ref{maintheorem} applies, and \eqref{mainformula} holds.

By \eqref{EM2}, for all $i'\geq 0$ we have
\begin{equation}
\label{etai'}
\begin{split}
\eta_{i'} &= \pm\sum_{n\gepm N} \frac1{s_n^{2(\delta + \theta) + i'}}
= \pm\sum_{n\gepm N} \frac1{s_n^{(1 + d i' + \deltabar + \thetabar)/d}}\\
&\equiv \pm a^{-(1 + d i' + \deltabar + \thetabar)/d} \frac{N^{-(d i' + \deltabar + \thetabar)}}{d i' + \deltabar + \thetabar} + \sum_{i = 1}^{\to\infty} \what f_i(d i' + \deltabar + \thetabar) N^{-(i + d i' + \deltabar + \thetabar)}\\
&= N^{-(\deltabar + \thetabar)} \begin{cases}
\frac{a^{-(1 + \thetabar)/d}}{\thetabar} + \sum_{i = 1}^{\to\infty} F_i^{(0)}(\thetabar) N^{-i} & i' = \deltabar = 0\\
\sum_{i = d i'}^{\to\infty} F_i^{(i')}(\thetabar) N^{-i} & d i' + \deltabar > 0,
\end{cases}
\end{split}
\end{equation}
where $F_i^{(i')}$ is analytic on $W \df \C_{> -t}$, where $t = \deltabar$ if $\deltabar > 0$ and $t = 1$ if $\deltabar = 0$. The reason first term of the case $i' = \deltabar = 0$ is positive is that in the case \eqref{goingup}, direct calculation shows that $\delta = \delta_E > 1/2d$ and thus $\deltabar > 0$. In what follows we assume that $\epsilon < t$, so that $B_\C(0,\epsilon) \subset W$.

It follows that
\begin{equation}
\label{etaI}
\eta_I \equiv N^{-\#(I)(\deltabar + \thetabar)} \sum_{i = d \Sigma(I)}^{\to\infty} F_i^{(I)}(\thetabar) N^{-i}
\end{equation}
with $F_i^{(I)}$ analytic on $W$, for all $I\in \MM(\N)$ if $\deltabar > 0$ and for all $I \in \MM(\N_{\geq 1})$ if $\deltabar = 0$.

We now consider the two cases (i) and (ii).
\begin{itemize}
\item[(i)] Suppose $\deltabar > 0$. Then since $E$ is regular and $\delta = \delta_E$, we have $\cbar = 0$ and thus \eqref{mainformula2} holds.

Note that $\etaratio \leq s_N^{-1} \lesssim N^{-d}$ and
\[
|\eta_0| \lesssim N^{-(\deltabar + \thetabar)} \leq N^{-\deltabar/2}
\]
for all sufficiently large $N$. Let $C$ denote the implied constant. Then for all $I,j$, by \eqref{cijkboundsv2} we have
\[
|c_{I,j} \, \eta_I \theta^j| \leq \epsilon^{-(\Sigma(I) + I(0) + j)} C^{\Sigma(I) + I(0)} N^{-(d\Sigma(I) + (\deltabar/2) I(0))} (\epsilon/2)^j
\]
assuming $N$ is sufficiently large. Now fix $p,q\in\N$. Since
\[
|c_{I,j} \, \eta_I \theta^j| \leq \begin{cases}
N^{-d p} 2^{-(\Sigma(I) + I(0) + j)} & \Sigma(I) \geq p \\
N^{-(\deltabar/2) q} 2^{-(\Sigma(I) + I(0) + j)} & I(0) \geq q
\end{cases}
\]
as long as $N$ is sufficiently large, by \eqref{mainformula2} we have
\[
\Xi = \sum_{\substack{I \in \MM(\N) \\ \Sigma(I) < p \\ I(0) < q}} \sum_{j = 0}^\infty c_{I,j} \, \eta_I \theta^j + O\left(N^{-d p} + N^{-(\deltabar/2) q}\right)
\]
and by \eqref{cijkboundsv2}, for each $I$ the function $F_I(\theta) \df \sum_j c_{I,j} \theta^j$ is analytic on (a neighborhood of) $B_\C(0,\epsilon/2) \subset W$. Combining with \eqref{etaI} and using the fact that $\eta_\0 = 1$, $c_{\0,0} = 0$, and $c_{\0,1} = -\w\chi$ gives
\begin{align*}
\Xi \equiv -\w\chi \theta + \sum_{j = 2}^\infty c_{\0,j} \theta^j + \sum_{i = 0}^{\to\infty} \sum_{k = 1}^{\to\infty} F_{i,k}(\theta) N^{-(i + k(\deltabar + \thetabar))} 
\end{align*}
where the functions 
\[
F_{i,k} = \sum_{\substack{I \in \MM(\N)\\ d\Sigma(I) \leq i\\ \#(I) = k}} F_I F_i^{(I)}
\] 
are analytic on (a neighborhood of) $B_\C(0,\epsilon/2)$. Thus, the equation \eqref{mainformula} (i.e. $\Xi = 0$) can be solved for $\theta$ as
\[
\theta \equiv \frac1{\w\chi}\left[\sum_{j = 2}^\infty c_{\0,j} \theta^j + \sum_{i = 0}^{\to\infty} \sum_{k = 1}^{\to\infty} F_{i,k}(\theta) \exp\big(-\hspace{-0.03in}k\,\thetabar\log(N)\big) N^{-(i + k\deltabar)} \right].
\]
Letting $\gamma = N^{\deltabar} \theta$, we have
\[
\gamma \equiv \sum_{j = 1}^\infty c_{\0,j + 1} \gamma^{j + 1} N^{-j \deltabar} + \sum_{i = 0}^{\to\infty} \sum_{k = 0}^{\to\infty} F_{i,k + 1}\left(\frac{1}{N^\deltabar} \gamma\right) \exp\left(-2d(k + 1)\frac{\log(N)}{N^\deltabar}\gamma\right) N^{-(i + k \deltabar)}.
\]
Solving for $\gamma$ in terms of $N^{-1}$, $N^{-\deltabar}$, and $N^{-\deltabar} \log(N)$, and then multiplying by $N^{-\deltabar}$ yields
\[
\theta = \sum_{i = 0}^{\to\infty} \sum_{k = 1}^{\to\infty} \sum_{j = 0}^{k - 1} c_{i,k,j} \frac{\log^j(N)}{N^{i + k\deltabar}}\:\cdot
\]
Note that $\w\chi c_{0,1,0} = F_{0,1}(0) = F_{\{0\}}(0) F_0^{(\{0\})}(0) c_{\{0\},0} = F_0^{(\{0\})} = \pm 1$, so $c_{0,1,0} = \pm 1/\w\chi$.

In the case \eqref{goingup}, if $F = \emptyset$, $s_n = n$, and $M = 1$, then $d = \delta = 1$ and thus $\deltabar = 1$, so
\[
\HD(F_{\leq N}) \equiv 1 + \sum_{i=1}^{\to\infty} \sum_{j = 0}^{i - 1} c_{i,j} \frac{\log^j(N)}{N^i}\cdot
\]
This proves Theorem \ref{theoremFleqN} from the introduction. See Appendix \ref{appendix1} for the computation of some of the coefficients $c_{i,j}$ in this case.

\item[(ii)] Suppose $\deltabar = 0$. Let $\xihat = a^{-(1 + \thetabar)/d} N^{-\thetabar}/\thetabar - \cbar$ and $\what\eta = \xi - \xihat = \eta_0 - a^{-(1 + \thetabar)/d} N^{-\thetabar}/\thetabar$. Then by \eqref{etaI} and \eqref{cijkbounds}, we have
\[
\Xi \equiv \sum_{i = 0}^{\to\infty} F_i(\thetabar,\xi,N^{-\thetabar}) N^{-i},
\]
where the $F_i$s are analytic on $B_\C(0,\epsilon/2)^3 \subset W\times \C^2$. (In fact, $F_i$ is a polynomial of degree $\leq i$ with respect to its third input.) Next, observe that by \eqref{etai'} we have
\[
\what\eta = \sum_{i = 1}^{\to\infty} F_i^{(0)}(\thetabar) N^{-(i + \thetabar)}.
\]
Since $\xi = \what\eta + \xihat$ and $N^{-\thetabar} = a^{(1 + \thetabar)/d} \thetabar (\cbar + \xihat\,)$, it follows that
\[
\Xi \equiv \sum_{i = 0}^{\to\infty} G_i(\thetabar,\xihat\,) N^{-i},
\]
where the $G_i$s are analytic on a neighborhood of $(0,0)$. By direct calculation, we have $G_0(0,0) = c_{\0,0,0} = 0$ and $(G_0)_{|2}(0,0) = c_{\0,0,1} = 1$. Thus, the equation \eqref{mainformula} (i.e. $\Xi = 0$) can be solved for $\xihat$:
\begin{equation}
\label{wxi}
\xihat = a^{-(1 + \thetabar)/d} \frac{N^{-\thetabar}}{\thetabar} - \cbar \equiv \sum_{i = 0}^{\to\infty} \what G_i(\thetabar) N^{-i},
\end{equation}
where the $\what G_i$s are analytic on a neighborhood of $0$, and by direct calculation, $G_i(0,\cdot) = 0$ and thus $\what G_i(0) = 0$ for all $i > 0$; similarly, $\what G_0(0) = 0$ since $G_i(0,0) = 0$. If $\cbar > 0$, then by solving for $N^{-\thetabar}$, using Lemma \ref{lemmathetageneral} to solve for $\thetabar$ (with $\aa = N^{-1}$), and finally solving for $\delta_N$ in terms of $\thetabar$, we complete the proof of \eqref{HDFgeqNgeneral}. (Note that Lemma \ref{lemmathetageneral} applies to asymptotic expansions because it applies to the estimates of which the asymptotic expansion is a limit.) The inequality $j \geq [i > 0]$ corresponds to the fact that $\what G_i(0) = 0$ for all $i$, since $\wbar\theta \sim \frac{\log\log(N)}{A d \log(N)}$.

So suppose that $\cbar = 0$. Then direct calculation gives $G_i(\cdot,0) = G_i(0,\cdot) = 0$ for all $i > 0$, and thus $\what G_i'(0) = 0$ for such $i$. On the other hand, since $S \neq \emptyset$ we have $\what G_0'(0) = -c_{\0,1,0} = \w\chi > 0$. So \eqref{wxi} becomes
\[
N^{-\thetabar} \equiv \thetabar^2\left(a^{1/d}\w\chi + \thetabar \sum_{i = 0}^{\to\infty} H_i(\thetabar) N^{-i}\right)
\]
for some analytic functions $H_i$. By taking the square root, using Lemma \ref{lemmathetageneral} to solve for $\thetabar/2$, and solving for $\delta_N$ in terms of $\thetabar$, we complete the proof of \eqref{HDFgeqNgeneral}. The inequality $j \geq [i > 0]$ corresponds to the fact that $\what G_i(0) = 0$ and $\what G_i'(0) = 0$ for all $i > 0$.

Suppose that $S = \emptyset$ and $a = 1$. Then since $s_n \in \N$ for all $n$, we have $b_1,\ldots,b_d \in \Q$ and thus for all $i$ we have $f_i(x) \in \Q[x]$ and $\what f_i(x) \in \Q(x)$, with notation as in \eqref{EM2}. It follows that all $F_i$s and $F_i^{(0)}$s are $\Q$-analytic and thus by Proposition \ref{propositionrational}, $G_i$ and $\what G_i$ are $\Q$-analytic. Since $\cbar = 1/\nu h = 1$, it follows that $c_{i,j,k,\ell} \in \Q$ for all $i,j,k,\ell$. Moreover, by \vanishinga\ and direct calculation we find that $\what G_0(\cdot) = 0$. Since $a = \cbar = 1$, it follows that
\begin{equation}
\label{equiv22}
N^{-\thetabar} \underset{X}{\equiv} \thetabar
\end{equation}
where $X = N^{-1} \thetabar^2$. The second part of Lemma \ref{lemmathetageneral} now demonstrates \eqref{HDFgeqN}.

We now wish to check that the values specified in Theorem \ref{theoremFgeqN} are correct. By \eqref{equiv22}, the coefficients $c_{k,\ell}$ are the same as the coefficients arising from the equation $N^{-\thetabar} = \thetabar$, so they can be computed explicitly from \eqref{alpharecursive} without any operator calculations, and we omit their calculation. On the other hand, since
\[
f(a,\thetabar,\xi) = 1 + \sum_{i = 0}^{\to\infty} \what G_i(\thetabar) a^i
\]
in Lemma \ref{lemmathetageneral}, we have
\begin{align*}
c_{1,1,-1,0} &= - (\log f)_{|12}(0,0,0)\\
&= -\what G_1'(0) = (G_1)_{|1}(0,0) = \Coeff\big(N^{-1} \thetabar,c_{\0,0,1} \what\eta \,+\, c_{\{1\},0,0} \eta_1 \,+\, c_{\0,1,0} \theta\big)\\
&= \Coeff\big(N^{-(1 + \thetabar)},\what\eta) = -P_1^+(0) = -1/2.
\end{align*}
This completes the proof.
\qedhere\end{itemize}
\end{proof}

\subsection{Quasi-geometric sequences}
\label{subsectionquasigeom}

We can ask what happens in the cases \eqref{goingup} and \eqref{goingdown} of Proposition \ref{propositionpolysequence} when instead of being a polynomial sequence, the sequence $(s_n)$ is an exponentially growing sequence such as a geometric sequence or the so-called \cite{Singh} Fibonacci sequence $s_n = s_{n - 1} + s_{n - 2}$ (initial conditions $s_0 = 0$, $s_1 = 1$). Recall that the $n$th term in the Fibonacci sequence is given by the formula
\begin{equation}
\label{fibonacci}
s_n = a \lambda^n (1 + b \rho^n)
\end{equation}
where $a = \frac1{\sqrt 5}$, $\lambda = \phi = \frac{1 + \sqrt 5}{2}$, $b = -1$, and $\rho = \wbar\phi / \phi = -\phi^{-2}$. The formula \eqref{fibonacci} can also be used to describe a geometric sequence, by letting $a$ and $\lambda$ be integers and $b = 0$.

\begin{proposition}
\label{propositionquasigeom}
Let $(s_n)$ be a sequence of positive integers defined by \eqref{fibonacci}, with $a > 0$, $\lambda > 1$, $b\in\R$, and $|\rho| < 1$. Let $F \subset \N$ be empty or strongly regular and disjoint from $\{s_M,s_{M+1},\ldots\}$ for some $M$, and let $E_N \to E$ be as in \eqref{goingup} or \eqref{goingdown} of Proposition \ref{propositionpolysequence}.
\begin{itemize}
\item[(i)] If $\delta > 0$, then
\begin{equation}
\label{quasigeom1}
\theta = \sum_{k = 1}^\infty \sum_{i = 0}^\infty \sum_{\tm = 0}^\infty \sum_{j = 0}^{k - 1} c_{k,i,\tm,j} (\lambda^{-(2\delta k + i)} \rho^\tm)^N N^j.
\end{equation}
\item[(ii)] If $\delta = 0$, then
\begin{equation}
\label{quasigeom2}
\theta = \frac{1}{A\log(\lambda)} \frac{1}{N} \left[\log(N) - \log\log(N) + \sum_{i = 0}^\infty \sum_{\tm = 0}^\infty \sum_{j = 0}^\infty \sum_{k = -j}^\infty \sum_{\ell = 0}^{j+k} c_{i,\tm,j,k,\ell} (\lambda^{-i} \rho^\tm)^N \frac{\log^\ell\log(N)}{N^j \log^k(N)}\right],
\end{equation}
where $A = 2$ if $\cbar > 0$ and $A = 1$ if $\cbar = 0$ (equivalently, $A = 2$ if $E = \emptyset$ and $A = 1$ if $\#(E) = 1$). Moreover, $c_{i,j,k,\tm} \in R \df \Q(\lambda,\rho)[a,\log(a),\log(\lambda),b,\log(\rho)]$ for all $i,j,k,\tm$.
\end{itemize}
Note that these formulas are exact and are not asymptotic expansions.

\end{proposition}
\begin{remark*}
Proposition \ref{propositionquasigeom} can be generalized to the setting where \eqref{fibonacci} is replaced by the formula
\[
s_n = a \lambda^n \left(1 + \sum_{i = 1}^d b_i \rho_i^n\right).
\]
This can be proven by making minor changes to the proof below.
\end{remark*}
\begin{proof}
As in the proof of Proposition \ref{propositionpolysequence}, for all $N$ we observe by direct calculation that the sets $\{s_M,\ldots,s_N\}$ and $\{s_N,s_{N + 1},\ldots\}$ are strongly regular, and thus since the union of two strongly regular sets is strongly regular, it follows that $E$ is empty or strongly regular and the $E_N$s are strongly regular, so Theorem \ref{maintheorem} applies and \eqref{mainformula} holds.

Let $\omega = 1$ if \eqref{goingup} holds and $\omega = 0$ if \eqref{goingdown} holds. Then $\eta_i = \pm f_\omega\big(2(\delta + \theta) + i\big)$, where
\begin{align*}
f_\omega(\alpha) &\df \sum_{n \geq N + \omega} \frac1{s_n^\alpha}
= \sum_{n\geq N + \omega} a^{-\alpha} \lambda^{-n\alpha} \sum_{\tm = 0}^\infty \binom{-\alpha}\tm b^\tm \rho^{\tm n}\\
&= \sum_{\tm = 0}^\infty a^{-\alpha} \binom{-\alpha}\tm b^\tm \sum_{n\geq N + \omega} (\lambda^{-\alpha} \rho^\tm)^n
= \sum_{\tm = 0}^\infty a^{-\alpha} \binom{-\alpha}\tm b^\tm \frac{(\lambda^{-\alpha} \rho^\tm)^{N + \omega}}{1 - \lambda^{-\alpha} \rho^\tm}\\
&= \sum_{\tm = 0}^\infty f_{t,\omega}(\alpha) (\lambda^{-\alpha} \rho^\tm)^N
\end{align*}
where for each $\tm > 0$, $f_{\tm,\omega}(\alpha) = a^{-\alpha} \binom{-\alpha}\tm b^\tm \frac{(\lambda^{-\alpha} \rho^\tm)^\omega}{1 - \lambda^{-\alpha} \rho^\tm}$ is analytic on $W = \C_{> \log(\rho)/\log(\lambda)}$; similarly, $f_{0,\omega}$ is analytic on $\C_{> 0}$. Note that since
\begin{align*}
\left|\binom z\tm\right| &\leq \frac{(|z| + \tm)^\tm}{\tm!} \leq \left(\frac{e (|z| + \tm)}{\tm}\right)^\tm\\
&\leq e^\tm \exp(|z|/\tm)^\tm = e^{\tm + |z|},
\end{align*}
we have $|f_{\tm,\omega}(\alpha)| \lesssim C^{\tm + |\alpha|}$ for some constant $C \geq 1$.

It follows that
\begin{align*}
\eta_i &= \pm f_\omega\big(2(\delta + \theta) + i\big)
= \lambda^{-2(\delta + \theta) N} \lambda^{-i N} \begin{cases}
\frac{a^{-2\theta}}{1 - \lambda^{-2\theta}} + \sum_{\tm = 1}^\infty f_{\tm,\omega}(2\theta) \rho^{\tm N} & \delta = i = 0\\
\sum_{\tm = 0}^\infty \pm f_{\tm,\omega}(2(\delta + \theta) + i) \rho^{\tm N} & 2\delta + i > 0
\end{cases}
\end{align*}
where as before, if $\delta = i = 0$ we use direct calculation to rule out the case \eqref{goingup}, allowing us to conclude that the first term of the top case is positive, as well as to reduce this term using the equality $\omega = 0$. Now since $|f_{\tm,\omega}(2(\delta + \theta) + i))| \lesssim C^{i + \tm}$ for all $\theta$ sufficiently close to $0$, there exists a ball $B$ centered at $\0$ satisfying the required bounds appearing Corollary \ref{corollaryanalytic} for the appropriate functions $F_i$, $F_*$, where
\[
\tt_N =  \begin{cases}
(\lambda^{-2(\delta + \theta) N},\lambda^{-N},\rho^N) & \text{ if $\delta > 0$}\\
(\lambda^{-2(\delta + \theta) N},\lambda^{-N},\rho^N,\xihat\,) & \text{ if $\delta = 0$}
\end{cases}
\]
Here $\xihat \df \lambda^{-2\theta N} \frac{a^{-2\theta}}{1 - \lambda^{-2\theta}} - \cbar$, and we use the fact that $\cbar = 0$ when $\delta > 0$ to write $\xi$ in terms of $\tt_N,\theta$ in that case. Thus by Corollary \ref{corollaryanalytic}, there exists a function $f$ analytic on $B$ such that $\Xi = f(\tt_N,\theta)$.
\begin{itemize}
\item[(i)] Suppose that $\delta > 0$. Then
\begin{align*}
\Xi &= F_{\0}(\theta) + \sum_{k = 1}^\infty \sum_{i = 0}^\infty \sum_{\tm = 0}^\infty F_{k,i,\tm}(\theta) \exp(-2k N \theta) \lambda^{-2k \delta N} \lambda^{-i N} \rho^{\tm N}\\
&= F_{\0}(\theta) + \sum_{k = 1}^\infty \sum_{i = 0}^\infty \sum_{\tm = 0}^\infty \sum_{j= 0}^\infty F_{k,i,\tm}(\theta) \frac{(-2k N \theta)^j}{j!} \lambda^{-2k \delta N} \lambda^{-i N} \rho^{\tm N}
\end{align*}
where $(F_{k,i,\tm})$ are analytic in a fixed neighborhood of $\0$, and $F_{\0}(0) = 0$ and $F_{\0}'(0) = c_{\0,1,0} = -\w\chi \neq 0$. Letting $\gamma = \lambda^{2\delta N} \theta$, solving for $\gamma$, and then dividing by $\lambda^{2\delta N}$ yields \eqref{quasigeom1}.

\item[(ii)] Now suppose that $\delta = 0$. Then since $\delta_E \leq \delta$, we have $\#(E)  \leq 1$ and thus we are in the case \eqref{goingdown}. Now let $\what\eta = \xi - \xihat = \eta_0 - a^{-2\theta} \lambda^{-2\theta N}/(1 - \lambda^{-2\theta})$. Since $\lambda^{-2\theta N} = (1 - \lambda^{-2\theta})(\cbar + \xihat\,)$ and $\xi = \what\eta + \xihat$, it follows that
\[
\Xi = \sum_{i = 0}^\infty \sum_{\tm = 0}^\infty F_{i,\tm}(\theta,\xihat\,) \lambda^{-iN} \rho^{\tm N}
\]
for some analytic functions $F_{i,\tm}$, and by direct calculation we have $F_{0,0}(0,0) = 0$ and $F_{0,0|2}(0,0) = c_{\0,0,1} = 1$. Thus we can solve for $\xihat$, which yields
\[
\frac{\lambda^{-2\theta N}}{2\theta\log(\lambda)} = \left(\frac{1 - \lambda^{-2\theta}}{2\theta\log(\lambda)}\right) \left(\cbar + \sum_{i = 0}^\infty \sum_{\tm = 0}^\infty F_{i,\tm}(\theta) (\lambda^{-i} \rho^\tm)^N\right).
\]
Note that $F_{i,\tm}(0) = 0$ for all $i,\tm$, since on a formal level, when $\theta = 0$, we have $\lambda^{-2\theta} = 0$ and thus $\eta_i = 0$ for all $i > 0$.

If $E = \emptyset$ then $\cbar = 1$, and thus Lemma \ref{lemmathetageneral} applies (with $\thetabar = 2\theta\log(\lambda)$, $B = e^{-N}$, and $\aa = (\lambda^{-N},\rho^N)$), yielding \eqref{quasigeom2} with $A = 2$. It can be shown using Proposition \ref{propositionrationalv2} that $c_{i,\tm,j,k,\ell} \in R$ for all $k,i,\tm$. This is because each function
\[
F_\tm^{(i)}(\theta) \df f_{\tm,0}(2\theta + i)
\]
is $R$-analytic, or $R$-meromorphic if $i = \tm = 0$.

If $\#(E) = 1$, then $\cbar = 0$, and since $\delta = 0$, by \vanishingc\ we have $c_{I,0,k} = 0$ for all $I,k$ with $I \neq \0$. Moreover, $F_{0,0}'(0) = -c_{\0,1,0} = \w\chi > 0$. So by taking the square root of the previous equation and letting $\thetabar = \theta\log(\lambda)$, Lemma \ref{lemmathetageneral} shows that \eqref{quasigeom2} holds with $A = 1$.
\qedhere\end{itemize}
\end{proof}

\subsection{Miscellaneous examples}

In the next two examples we consider relatively simple sequences $(E_N)$ of two-element sets with dimension tending to zero, one where both generators tend to zero and another where one of the generators is fixed. These examples illustrate the variety of behavior that can occur when solving \eqref{mainformula}.

\begin{example}
Let $E_N = \{N,N+1\} \to E = \emptyset$. Then
\begin{equation}
\label{thetaNN1}
\theta = \frac{1}{2\log(N)} \left[\log(\phi) + \sum_{i = 2}^\infty \sum_{j = 1}^{i - 1} \frac{c_{i,j}}{N^i \log^j(N)}\right]
\end{equation}
for all $N$ sufficiently large, where $c_{i,j} \in \Q[\phi,\log(\phi)]$ for all $i,j$, where $\phi \df (1 + \sqrt 5)/2$. In particular, 
\[
c_{2,1} = \frac{4\phi^{-1} \log(\phi)}{1 + \phi^{-2}}\cdot
\]
\end{example}
\begin{proof}
Since $E$ is empty and the $E_N$s are strongly regular, Theorem \ref{maintheorem} applies and thus \eqref{mainformula} holds. Now $\cbar = 1$ and thus for all $i$,
\[
\eta_i = \frac{1}{N^{i + 2\theta}} + \frac{1}{(N+1)^{i + 2\theta}} = N^{-(i + 2\theta)} \left[1 + \frac{1}{1 + N^{-(i + 2\theta)}}\right] = \begin{cases}
\cbar + F_*(N^{-2\theta} - \phi^{-1}) & i = 0\\
F(N^{-(i + 2\theta)}) & i > 0,
\end{cases}
\]
where $F_*,F$ are $\Q[\phi]$-analytic functions with $F_*(0) = F(0) = 0$, $F_*'(0) = 1 + \phi^{-2} > 0$, and $F'(0) = 2$. So $\xi = F_*(N^{-2\theta} - \phi^{-1})$, and thus by Corollary \ref{corollaryanalytic}, for all $N$ sufficiently large we have
\[
\Xi = G(N^{-1},\theta,N^{-2\theta} - \phi^{-1})
\]
where $G$ is analytic in a neighborhood of $\0$, $G(\0) = 0$, and $G_{|3}(\0) = F_*'(0) c_{\0,0,1} = 1 + \phi^{-2}$. Moreover, by either Proposition \ref{propositionrationalv2} or Proposition \ref{propositionrational}, $G$ is $\Q[\phi]$-analytic. Since $S = \emptyset$ and $\delta = 0$, by Proposition \ref{propositionvanishing}, solving $\Xi = 0$ for $N^{-2\theta} - \phi^{-1}$ and then rearranging yields
\[
N^{-2\theta} = \phi^{-1} + N^{-2} \theta H(N^{-1},N^{-1}\theta)
\]
where $H$ is $\Q[\phi]$-analytic. Taking logarithms and dividing by $-2\log(N)$, we see that
\[
\theta = \frac1{2\log(N)}\big[\log(\phi) - \log(1 + \phi N^{-2} \theta H(N^{-1},N^{-1} \theta))\big].
\]
By the implicit function theorem applied to the above equation treating $\theta$, $N^{-1}$, and $\log(N)^{-1}$ as the basic variables, solving for $\theta$ demonstrates \eqref{thetaNN1}. To see why the bounds on $i$ and $j$ follow, note that if we substitute $\alpha = 2\theta\log(N) - \log(\phi)$, then in the resulting formula all terms are of the form $N^{-i} \log(N)^{-j} \alpha^k$, with $1 \leq j \leq i-1$. An induction argument shows that these bounds hold for all terms of the series representing $\alpha$.

Now, $c_{2,1} = -\phi\frac{\log(\phi)}{2} H(0,0)$, and
\[
H(0,0) = \frac{G_{|112}(\0)}{G_{|3}(\0)} = \frac{(\phi^{-1} F'(0))^2}{F_*'(0)} \cdot \frac{c_{2\{1\},1,0}}{c_{\0,0,1}} = \frac{4\phi^{-2}}{1 + \phi^{-2}} \nu \beta_{1,0} \beta_{1,1} h,
\]
where in the last step we use Proposition \ref{propositionvanishing} to ignore the other possible contributions to $c_{2\{1\},1,0}$. The calculation is completed by observing that
\begin{align*}
\beta_{1,1} h(x) &= \Coeff\big(b^1 \theta^1, e^{\theta v(b,x)}\big) = \Coeff(b^1,v(b,x)) = -2 x\\
\nu \beta_{1,0} \beta_{1,1} h &= -2\Coeff\big(b^1 \theta^0, e^{\theta v(b,0)} u_b(0)\big)
= -2\Coeff\big(b^1,b) = -2.
\qedhere\end{align*}
\end{proof}

\begin{example}
Let $E_N = \{1,N\} \to E = \{1\}$. Then
\begin{equation}
\label{HDF1N}
\theta = \frac1{2\log(N)} \left[\log\log(N) - \log\log\log(N) + \sum_{i = 0}^\infty \sum_{j = [i > 0]}^\infty \sum_{k = -j}^\infty \sum_{\ell = 0}^{j + k} c_{i,j,k,\ell} \frac{\log^\ell\log\log(N)}{N^i \log^j(N) \log^k\log(N)} \right]
\end{equation}
for all $N$ sufficiently large, and $c_{0,0,0,0} = -\log\log(\phi) > 0$. Notice that in the second summation, we use the Iverson bracket notation: $[\Phi] = 1$ when $\Phi$ is true and $[\Phi] = 0$ when $\Phi$ is false. 
\end{example}
\begin{proof}
Since $E$ and the $E_N$s are strongly regular, Theorem \ref{maintheorem} applies and thus \eqref{mainformula} holds. Now
\[
\eta_i = \sum_{\an\in E_N - E} \frac{1}{\an^{i + 2(\delta + \theta)}} = N^{-i} N^{-2\theta}
\]
and since $P(E,\delta) = P(E,0) = \log\#(E) = 0$, we have $\cbar = 0$. Thus by Corollary \ref{corollaryanalytic} and \vanishingc,
\[
\Xi = \theta F_0(\theta) + N^{-2\theta} F_1(N^{-2\theta}) + N^{-2\theta} \theta F_2(N^{-1},\theta,N^{-2\theta})
\]
where $F_0,F_1,F_2$ are analytic in a neighborhood of $\0$ with $F_0(\0) = c_{\0,1,0} = \w\chi =\chi = 2\log(\phi) > 0$ and $F_1(\0) = c_{\0,0,1} = 1 > 0$. Solving $\Xi = 0$ for $N^{-2\theta}$ gives
\[
N^{-2\theta} = \theta \big(2\log(\phi) + \theta g(N^{-1},\theta)\big) = 2\theta \big(\log(\phi) + (1/2) \theta g(N^{-1},\theta)\big)
\]
where $g$ is analytic in a neighborhood of $\0$. Thus Lemma \ref{lemmathetageneral} applies and we have \eqref{HDF1N}, with $c_{0,0,0,0} = -\log\log(\phi)$. The $i=j=0$ case of the summation is ruled out because of the factor $\theta$ appearing in $(1/2)\theta g(N^{-1},\theta)$ above.
\end{proof}

\section{Directions to further research}
\label{sectionopen}

We conclude the main body of our paper by presenting a sample of problems and research directions, which we hope will partially illustrate the wide scope awaiting future exploration.

We speculate that the key ideas behind our our basic perturbation result, Theorem \ref{theoremoperatorequation}, might apply more generally. For instance, it would be interesting to leverage our perturbation theorem or some variant thereof to analyze other functionals that arise in the study of dynamical systems and stochastic processes, and perturbations thereof.

With a view to developing asymptotic expansions that lie beyond the scope of this paper, here are three concrete CIFSes for which our methods do not directly apply:
\begin{itemize}
\item The co-Cantor similarity IFS described in Example \ref{examplecocantor}.
\item The CIFS $u_n: x \mapsto (1 + x^n)2^{-n}$ for $n \in \N$, described in Example \ref{exampleNotPACIFS}.
\item The prime alphabet Gauss CIFS: $u_p: x \mapsto (p + x)^{-1}$ for primes $p$.
\end{itemize}
It may be the case that it is impossible to develop an asymptotic expansion for the last example above, which is an OSC PACIFS unlike the previous two examples.

It is natural to attempt a generalization of our results in \6\ref{subsectionquasigeom} where we studied alphabets that were quasi-geometric (e.g. Fibonacci) sequences to the broader class of \emph{constant-recursive sequences}, i.e. sequences satisfying a linear recurrence with constant coefficients. One could consider holonomic or $P$-recursive sequences, $k$-automatic sequences, or $k$-regular sequences. It would be interesting to precipitate connections with ideas familiar to {\it analytic combinatorics} and {\it analysis of algorithms} communities. Perhaps investigating the coefficient numerology in the asymptotic expansions studied in this paper will lead more directly to such links.

Though the main applications in this article have emphasized the approximation of real numbers by rationals using the simple continued fraction algorithm, there exist several avenues of active research inspired by this particular seam with a multitude of surprising interactions with dynamical systems and number theory. Dajani and Kraaikamp's Carus monograph \cite{DajaniKraaikamp} is a beautifully written introduction to the ergodic theory of several numeration schemes and continued fraction algorithms; and for higher-dimensional variants see \cite{Schweiger, Schweiger2, ArnouxSchmidt, Berthe}. 
It would be interesting to leverage our techniques to study analogues of our results in any of these settings.

For instance, one could focus on any one of several existing families of piecewise smooth expanding maps of the interval with infinitely many branches, e.g. those arising from the family of \emph{Japanese continued fractions}, named for Hitoshi Nakada, Shunji Ito, and Shigeru Tanaka -- see \cite{CesarattoVallee2,CarminatiTiozzo2,KSS2,BDV3,DHKM} for these systems and some variations. Furthermore, continued fraction algorithms arising from study of geodesic flows on negatively curved surfaces \cite{KatokUgarcovici2,BocaMerriman}, and continued fraction expansions over the field of Laurent series \cite{Schmidt9,Wu3,BertheNakada,HWWY} would be natural environs to investigate analogues of our results.

To conclude, here are three scenarios to reconnoiter in the higher-dimensional setting:
\begin{itemize}
\item Find analogues of our results for \emph{complex continued fractions}, where one studies the complex analogue of the Gauss CIFS $\{ u_{n}: x \mapsto (n+x)^{-1} \;\text{for}\; n \in \N \}$ on the unit interval, by replacing $\N$ with the Gaussian integers with positive real part, and the unit interval with $B(1/2,1/2) \subset \C$. Such systems may be profitably studied within the broad framework of conformal graph directed Markov systems (CGDMSes), see \cite{CLU} and its references.
\item CIFS and CGDMS limit sets model several fractals that arise from Sullivan's dictionary \cite{McMullen_classification, Sullivan_conformal_dynamical} (see also \cite[Table 1]{DSU_rigidity}), which translates between the study of Julia sets associated with holomorphic and meromorphic iteration, and Kleinian limit sets associated with actions of discrete subgroups of isometries of hyperbolic (negatively curved) spaces. Exploring analogues of our results within this broad framework would be very interesting.
\item Find analogues of our results beyond the conformal setting, e.g. for infinite self-affine IFSes studied by Jurga \cite{Jurga}, or the broad class of examples studied by Reeve \cite{Reeve}.
\end{itemize}

\appendix
\section{Computation of coefficients}
\label{appendix1}

Although we do not state precise formulas for the coefficients $c_{I,j,k}$ and $c_{I,j}$ appearing in Theorem \ref{maintheorem}, our proofs do facilitate the construction of such formulas. Namely, Lemmas \ref{lemmaalphaf} and \ref{lemmabetaf} allow us to write $\alpha$ and $\beta$ in terms of the secondary operators $\alpha_j$ and $\beta_{i,j}$. Plugging these into \eqref{operatorequation} gives the coefficients for the power series $\Xi$ in terms of $(\eta_i)$, $\theta$, and $\xi$.

To illustrate this process, we compute Hensley's coefficients $c_{1,0}$ and $c_{2,1}$ for the sequence of systems $E_N = \{1,\ldots,N\} \to E = \N$, and we show that the formula for $c_{2,0}$ in \eqref{HDFleqN} involves Apery's constant $\zeta(3)$ as well as the expressions
\begin{align}
\label{Qterms}
\mu M_\phi &Q L M_\phi g, &
\mu M_\phi &Q h, &
\nu &Q L M_\phi g, &
\nu &Q h,
\end{align}
where the notation is as in \6\ref{sectiongauss}, and $M_\phi$ denotes multiplication by the function
\[
\phi(x) = 2\log(x).
\]
Note that it appears to be impossible to rewrite even the simplest of these expressions, $\nu Q h = Q\one(0) = 1 + \sum_{n\geq 1} (L^n \one(0) - 1/\log(2))$, in closed form. However, all of the expressions can be approximated with arbitrary accuracy.

To this end, in what follows we let $X = N^{-1} \log(N)$. Note that since $\theta = O(N^{-1})$, we have $c_{I,j} \, \eta_I \theta^j \equiv_{X^k} 0$ whenever $\#(I) + \Sigma(I) + j \geq k$. In particular, the second-order approximation of \eqref{operatorequation} is
\[
\mu\alpha g + \mu\beta g + \mu\Delta Q \Delta g \underset{X^3}{\equiv} \Xi = 0.
\]
Next we observe that $\phi \circ u_b(x) = \log|u_b'(x)|$ for all $b,x$. Thus since $\delta = 1$, for all $j \geq 1$ we have
\[
\alpha_j f(x) = \frac1{j!} \sum_{b \in S} |u_b'(x)| \log^j|u_b'(x)| \; f\circ u_b(x) = \frac1{j!} \sum_{b\in S} |u_b'(x)| (\phi^j f)\circ u_b(x) = \frac1{j!} L M_\phi^j f(x)
\]
so
\begin{align*}
\mu \alpha_j g &= \frac1{j!} \mu L M_\phi^j g
= \frac1{j!} \mu M_\phi^j g\\
&= \frac{2^j}{j!} \int_0^1 \frac{\log^j(x)}{1 + x} \;\dee x
= \frac{2^j}{j!} \sum_{n = 0}^\infty (-1)^n \int_0^1 x^n \log^j(x) \;\dee x \\
&= 2^j \sum_{n = 0}^\infty (-1)^n \sum_{j' = 0}^j \frac{(-1)^{j'}}{(n+1)^{j' + 1} (j - j')!} \big[x^{n+1} \log^{j - j'}(x)\big]_{x=0}^1 \\
&= (-1)^j 2^j \sum_{n = 0}^\infty \frac{(-1)^n}{(n+1)^{j+1}}
= (-1)^j 2^j (1 - 2^{-j}) \zeta(j+1)\\
&= (-1)^j (2^j - 1) \zeta(j + 1)
\end{align*}
and thus
\begin{align*}
\mu\alpha g
&= \sum_{j=1}^\infty \theta^j \mu \alpha_j g
\underset{X^3}{\equiv} \theta \mu \alpha_1 g + \theta^2 \mu \alpha_2 g
= -\zeta(2)\theta + 3 \zeta(3) \theta^2.
\end{align*}
On the other hand, by direct computation\Footnote{The second equality is guaranteed by \vanishinga.} we have
\begin{align*}
\mu \beta_{0,0} g &= 1, &
\mu \beta_{0,1} g &= 0,&
\mu \beta_{1,0} g &= \int_0^1 -2x-1 \;\dee x = -2
\end{align*}
and
\begin{align*}
\eta_0 &\underset{X^3}{\equiv} -\left(\frac{N^{-(1 + 2\theta)}}{1 + 2\theta} - \frac{N^{-(2 + 2\theta)}}{2}\right)
\underset{X^3}{\equiv} -\frac1N + \frac{2\theta\log(N)}{N} + \frac{2\theta}{N} + \frac{1}{2N^2} \underset{X^2}{\equiv} -\frac1N,\\
\eta_1 &\underset{X^3}{\equiv} -\frac{N^{-(2 + 2\theta)}}{2 + 2\theta} \underset{X^3}{\equiv} -\frac{1}{2N^2}
\end{align*}
so
\begin{align*}
\mu\beta g
&\underset{X^3}{\equiv} \eta_0 (\mu\beta_{0,0} g + \theta \mu\beta_{0,1} g) + \eta_1 (\mu \beta_{1,0} g)\\
&\underset{X^3}{\equiv} \left(-\frac1N + \frac{2\theta\log(N)}{N} + \frac{2\theta}N + \frac1{2N^2}\right) (1) + \left(-\frac1{2N^2}\right) (-2)
\end{align*}
Finally, since $\beta_{0,0} = h \nu$, $\mu h = \nu g = 1$, and $\mu L = L$, we have
\begin{align*}
\mu\Delta Q \Delta g
&\underset{X^3}{\equiv} (\theta \mu M_\phi + \eta_0 \nu) Q (\theta L M_\phi g + \eta_0 h).
\end{align*}
Next we compute the first-order approximation of $\Xi$:
\[
\Xi  \underset{X^2}{\equiv} -\zeta(2)\theta + (-1/N)
\]
and so setting $\Xi = 0$ yields $\theta \equiv_{X^2} -1/\zeta(2)N$, giving Hensley's first coefficient 
\[
c_{1,0} = -1/\zeta(2) = -6/\pi^2.
\] 
Plugging this into the above formulas gives
\begin{align*}
\theta = \frac1{\zeta(2)}\big(\Xi + \zeta(2)\theta\big) \underset{X^3}{\equiv} \frac1{\zeta(2)} &\left[- \frac1N - \frac{2}{\zeta(2)} \frac{\log(N)}{N^2} + \left(\frac{3}{2} - \frac{2}{\zeta(2)} + \frac{3\zeta(3)}{\zeta^2(2)}\right)\frac{1}{N^2}\right.\\
&+ \left.\left(\frac1{\zeta(2)} \mu M_\phi + \nu\right) Q \left(\frac1{\zeta(2)} L M_\phi g + h\right)\frac{1}{N^2}\right].
\end{align*}
This formula gives Hensley's second coefficient $c_{2,1} = -2/\zeta^2(2) = -72/\pi^4$: the next coefficient is
\begin{equation}
\label{c20}
c_{2,0} = \frac{3}{2} - \frac{2}{\zeta(2)} + \frac{3\zeta(3)}{\zeta^2(2)} + \left(\frac1{\zeta(2)} \mu M_\phi + \nu\right) Q \left(\frac1{\zeta(2)} L M_\phi g + h\right)
\end{equation}
Notice that the four terms of \eqref{Qterms} all appear in this formula.

\subsection{Some further coefficients}
\label{appendix2}

It turns out to be possible to compute the coefficients $c_{i,i - 1}$ directly without dealing with any coefficients $c_{i,j}$ such that $j \leq i - 2$. Namely, let us write $A \equiv_p B$ if
\[
B - A \equiv N^{-p} \sum_{i = 0}^{\to\infty} \sum_{j = 0}^i c_{i,j} \frac{\log^j(N)}{N^i}
\]
for some coefficients $c_{i,j}$. We can think of this as saying that $B - A$ is ``formally $O(N^{-p})$'', in a sense where $N^{-1} \log(N)$ is considered ``small'' but $N^{-1} \log^2(N)$ is not considered ``small''.

Now \eqref{operatorequation} becomes
\[
\mu\alpha g + \mu\beta g \underset{2}{\equiv} \Xi = 0.
\]
Moreover, similarly to before we have $\mu\alpha g \equiv_2 \theta \mu\alpha_1 g = -\zeta(2)\theta$. On the other hand,
\[
\mu\beta g \underset{2}{\equiv} \eta_0 \mu \beta_{0,0} g = \eta_0 \underset{2}{\equiv} -N^{-1 - 2\theta}
\]
by the Euler-Maclaurin formula. Thus, \eqref{operatorequation} becomes
\[
\zeta(2)\theta \underset{2}{\equiv} -N^{-(1 + 2\theta)} = -N^{-1} \exp(-2\theta \log(N)).
\]
So we have $-\zeta(2) N\theta \equiv_2 F(2\log(N)/\zeta(2)N)$, where
\begin{equation}
\label{aidef}
F(x) = \sum_{j = 0}^\infty a_j x^j \;\;\;\; \text{ satisfies} \;\;\;\;
F(x) = \exp(x F(x)),
\end{equation}
i.e. $F$ is the inverse of $y\mapsto \log(y)/y$ defined in a neighborhood of $0$ and sending $0$ to $1$. It follows that $c_{i,i - 1} = -(2^{i-1}/\zeta^i(2)) a_{i-1}$. 

To compute $a_j$, we first recall Cayley's formula: the number of spanning trees on $i$ points is $T_i = i^{i - 2}$.  To produce a recursive formula for $(T_i)$, observe that to define a spanning tree on $i$ points, you need to define (a) a partition of the set of $i-1$ points, (b) spanning trees on each element of the partition, and (c) a root node in each of these spanning trees to connect to the final node to form the overall tree. 

Now compare the recursive formulas for $(a_j)$ and $(T_i)$:
\begin{align*}
a_j &= \sum_{n = 0}^\infty \frac1{n!} \sum_{\substack{t\in \N^n \\ |t| + n = j}} \prod_{k = 1}^n a_{t_k}
= \sum_{P\in \PP_j} \prod_{A\in P} \#(A)! a_{\#(A) - 1}\\
T_i &= \sum_{P\in \PP_{i - 1}} \prod_{A\in P} \#(A) T_{\#(A)}
\end{align*}
where $\PP_n$ is the set of all partitions of $\{1,\ldots,n\}$. It follows that
\[
a_j = \frac{T_{j + 1}}{j!} = \frac{(j + 1)^{j - 1}}{j!} \cdot
\]
So
\[
c_{i,i-1} = -\frac{2^{i-1}}{\zeta^i(2)} \frac{i^{i-2}}{(i-1)!},
\]
which is equivalent to \eqref{cii-1} from the introduction.

\section{Definition of a conformal iterated function system (CIFS)}
\label{appendixCIFS}

We recall the definition of a conformal iterated function system (CIFS) due to Mauldin--Urba\'nski.
\begin{definition}[Cf. {\cite[p.108-110]{MauldinUrbanski1}}]
\label{definitionCIFS}
Fix $d\in\N$. A collection of maps $(u_a)_{a\in E}$ is called a \emph{conformal iterated function system (CIFS)} on $\R^d$ if:
\begin{enumerate}[1.]
\item $E$ is a countable (finite or infinite) index set;
\item $X\subset\R^d$ is a nonempty compact set which is equal to the closure of its interior;
\item For all $a\in E$, $u_a(X) \subset X$;
\item (Cone condition)
\[
\inf_{\xx\in X, r\in (0,1)} \frac{\lambda(X\cap B(\xx,r))}{r^d} > 0,
\]
where $\lambda$ denotes Lebesgue measure on $\R^d$;
\item $V\subset\R^d$ is an open connected bounded set such that $\dist(X,\R^d\butnot V) > 0$;
\item For each $a\in E$, $u_a$ is a conformal homeomorphism from $V$ to an open subset of $V$;
\item (Uniform contraction) $\sup_{a\in E} \sup |u_a'| < 1$, and if $E$ is infinite, $\lim_{a\in E} \sup |u_a'| = 0$;
\item (Bounded distortion property) For all $n\in\N$, $\omega\in E^n$, and $\xx,\yy\in V$,
\begin{equation}
\label{BD2}
|u_\omega'(\xx)| \asymp_\times |u_\omega'(\yy)|,
\end{equation}
where
\[
u_\omega = u_{\omega_1}\circ\cdots\circ u_{\omega_n}.
\]
\end{enumerate}
The CIFS is called an \emph{OSC CIFS} if in addition it satisfies the \emph{open set condition (OSC)}, i.e. if the collection $(u_a(\Int(X)))_{a\in E}$ is disjoint. 
\end{definition}

\hrulefill

{\bf Acknowledgements.} This research began on 12$^{th}$ March 2018 when the authors met at the American Institute of Mathematics via their SQuaRE program. We thank the institute and their staff for their hospitality and excellent working conditions. In particular, we thank Estelle Basor for her continued encouragement and support.
The first-named author was supported in part by a 2017-2018 Faculty Research Grant from the University of Wisconsin-La Crosse. 
He thanks the scientific and organizing committees of the \href{http://sites.math.u-pem.fr/oneworld-fractals/}{\it One-world Fractals and Related Fields} seminar, in particular St\'ephane Seuret and Julien Barral, for the opportunity to speak about this work at his first virtual research lecture. 
The third-named author was supported in part by the EPSRC Programme Grant EP/J018260/1, and also in part by a Royal Society University Research Fellowship, URF\tbs R1\tbs180649. The fourth-named author was supported in part by a Simons Foundation Grant 581668. We thank the referee for their comments and suggestions to help improve the exposition. 

{\bf Data availability statement.} Data sharing not applicable to this article as no datasets were generated or analysed during the current study.

\bibliographystyle{amsplain}

\bibliography{bibliography}

\end{document}